\documentclass[oneside,english]{amsart}
\usepackage[T1]{fontenc}
\usepackage[latin9]{inputenc}
\usepackage{textcomp}
\usepackage{amstext}
\usepackage{amsthm}
\usepackage{amssymb}
\usepackage{graphicx}
\usepackage[all]{xy}

\makeatletter
\numberwithin{equation}{section}
\numberwithin{figure}{section}
\theoremstyle{plain}
\newtheorem*{thm*}{\protect\theoremname}
\theoremstyle{plain}
\newtheorem{thm}{\protect\theoremname}
\theoremstyle{definition}
\newtheorem{defn}[thm]{\protect\definitionname}
\theoremstyle{plain}
\newtheorem{prop}[thm]{\protect\propositionname}
\theoremstyle{plain}
\newtheorem{lem}[thm]{\protect\lemmaname}
\theoremstyle{plain}
\newtheorem{cor}[thm]{\protect\corollaryname}
\theoremstyle{definition}
\newtheorem{example}[thm]{\protect\examplename}
\theoremstyle{remark}
\newtheorem{rem}[thm]{\protect\remarkname}

\makeatother

\usepackage{babel}
\providecommand{\corollaryname}{Corollary}
\providecommand{\definitionname}{Definition}
\providecommand{\examplename}{Example}
\providecommand{\lemmaname}{Lemma}
\providecommand{\propositionname}{Proposition}
\providecommand{\remarkname}{Remark}
\providecommand{\theoremname}{Theorem}

\begin{document}

\title{Characteristic classes for TC structures}

\author{Mauricio Cepeda Davila}
\begin{abstract}
In this article we study the construction of characteristic classes
for principal $G$-bundles equipped with an additional structure called
transitionally commutative structure (TC structure). These structures
classify, up to homotopy, possible trivializations of a principal
$G$-bundle, such that the induced cocycle have functions that commute
in the intersections of their domains. We focus mainly on the cases
where the structural group G equals SU(n), U(n) or Sp(n). Our approach
is an algebraic-geometric construction that relies on the so called
power maps defined on the space $B_{\mathrm{com}}G$, the classifying
space for commutativity in the group G.
\end{abstract}

\thanks{The author was sponsored by the Colombian Ministry of Sciences (previously
Colciencias) under the public sponsorship act 647 of 2014 for national
doctoral programs. Where the results presented here are part of his
PhD thesis, supervised by José Manuel Gómez Guerra at the National
University of Colombia at Medellín. }
\maketitle

\section{Introduction}

Suppose that $G$ is a Lie group and consider the set of $n$-tuples
with commuting elements which can be identified with $\mathrm{Hom}\left(\mathbb{Z}^{n},G\right)$.
Adem, Cohen and Torres-Giese (see \cite{key-17}) showed that $\left\{ \mathrm{Hom}\left(\mathbb{Z}^{n},G\right)\right\} _{n\geq0}$
can be endowed with a simplicial structure whose geometric realization
is denoted by $B_{\mathrm{com}}G$. Additionally, consider a principle
$G$-bundle over a compact Hausdorff space $M$, with a classifying
function $g:M\rightarrow BG$ and trivializations with associated
cocycle $\left\{ \rho_{ij}\right\} $. Suppose then that the cocycles
commute with each other in the intersection of their domains, i.e.
$\rho_{ij}\ensuremath{\cdot}\rho_{jk}=\rho_{jk}\rho_{ij}.$ Adem and
Gomez showed in \cite{key-1} that the previous commutativity condition
holds if only if there is a lifting, up tho homotopy of the classifying
map $g$; that is a commutative diagram up to homotopy as the one
shown below, 
\[
\xymatrix{ & B_{\mathrm{com}}G\ar[d]^{\iota}\\
M\ar[ru]^{f}\ar[r]^{g} & BG.
}
\]

The existence of such lifting is what we call a \textbf{transitionally
commutative (TC) structure} on a principal $G$-bundle. Where we say
that two TC structures $f_{1},f_{2}:M\rightarrow B_{\mathrm{com}}G$
are equivalent if the functions are homotopic. TC structures are meant
to classify the different ways a principle bundle can have commutative
cocycles, up to homotopy. 

The interest in studying the spaces $\mathrm{Hom}\left(\mathbb{Z}^{n},G\right)$
arises from the study of moduli spaces of flat bundles, which are
important for Quantum field theories such as the Yang-Mills and Chern-
Simons theories. In particular, when the base space is the torus $\left(S^{1}\right)^{n}$
and the structural group is a compact Lie group $G$, the moduli spaces
of flat bundles can be identified with $\mathrm{Hom}\left(\mathbb{Z}^{n},G\right)/G$,
where $G$ acts under conjugation. 

Mathematically speaking, the theory of commuting tuples is interesting
in its own right. For example, Adem and Gomez defined in \cite{key-1}
the commutative K-theory of a finite CW-complex X to be $K_{\mathrm{com}}X:=Gr\left(\mathrm{Vect}_{\mathrm{com}}\left(X\right)\right)$,
where $Gr$ denotes the Grothendieck construction and $\mathrm{Vect}_{\mathrm{com}}\left(X\right)$
is the set of equivalence classes of vector bundles over $X$ with
commuting cocycles. Later on Adem, Gómez, Lind and Tillman introduce
the notion of $q$-nilpotent $K$-theory of a CW-complex $X$ for
any $q\ge2$, which extends the notion of commutative $K$-theory
defined by Adem and Gomez, and show that it is represented by $\mathbb{Z}\times B(q,U)$,
were $B(q,U)$ is the $q$-th term of a filtration of the infinite
loop space $BU$. (See \cite{key-18}.)

In this article we define and develop characteristic classes for TC
structures and mainly, we develop an algebraic-geometric method to
use Chern-Weil theory to compute them. To start we will see that there
is a one to one correspondence between characteristic classes and
elements of $H^{\text{*}}\left(B_{com}G,\mathbb{R}\right)$. We then
use the description of $H^{\text{*}}\left(B_{com}G,\mathbb{R}\right)$
given in \cite{key-1} for which we exhibit a set of algebraic generators.
Then we show how we can use Chern-Weil theory to compute the characteristic
classes associated to those generators. We do this for $G$ equal
to either $U(n)$, $SU(n)$ or $\mathrm{Sp}(n)$ for the following
reasons. 

In general the spaces $\mathrm{Hom}\left(\mathbb{Z}^{n},G\right)$
are not path connected, so the simplicial construction can be reduced
to consider the path connected components containing the identity
tuple $\left(1,1,\ldots,1\right)$. These connected components are
denoted by $\mathrm{Hom}\left(\mathbb{Z}^{n},G\right)_{1}$, and the
geometric realization of them is denoted by $B_{\mathrm{com}}G_{1}$.
Adem and Gomez showed in Proposition 7.1 of \cite{key-1} that the
cohomology with real coefficients of $B_{\mathrm{com}}G_{1}$ is isomorphic
to 
\begin{equation}
\left(H^{*}\left(BT,\mathbb{R}\right)\otimes H^{*}\left(BT,\mathbb{R}\right)\right)^{W}/J,\label{eq:Cohomology of BcomG}
\end{equation}
where $T\subset G$ is a maximal tori, $W$ is the Weil group acting
diagonally, and $J$ is the ideal generated by elements of the form
$p\left(x\right)\otimes1$ with $p\left(x\right)$ an $W$-invariant
polynomial of positive degree. 

Additionally, Adem and Cohen showed in Corollary 2.4 of \cite{key-16}
that $\mathrm{Hom}\left(\mathbb{Z}^{n},G\right)$ is path connected
when $G$ is either $U(n)$, $SU(n)$ or $\mathrm{Sp}(n)$. For them
then Expression \ref{eq:Cohomology of BcomG} describes the cohomology
of all $B_{\mathrm{com}}G$. To obtain the generators of this cohomology,
we first consider the natural inclusion $\mathrm{Hom}\left(\mathbb{Z}^{n},G\right)\subseteq G^{n}$.
The inclusion induces a simplicial map between the simplicial structure
of $\left\{ \mathrm{Hom}\left(\mathbb{Z}^{n},G\right)\right\} _{n\geq0}$
and the bar construction for the classifying space of $G$, $BG$.
This in turn gives rise to a map 
\[
\iota:H^{\text{*}}\left(BG,\mathbb{R}\right)\rightarrow H^{\text{*}}\left(B_{com}G,\mathbb{R}\right).
\]
Secondly, we need to consider the assignments 
\begin{align*}
\mathrm{Hom}\left(\mathbb{Z}^{n},G\right) & \rightarrow\mathrm{Hom}\left(\mathbb{Z}^{n},G\right)\\
\left(g_{1},\ldots,g_{n}\right) & \mapsto\left(g_{1}^{k},\ldots,g_{n}^{k}\right).
\end{align*}
These assignments give rise to simplicial maps, allowing us to obtain
the power maps 
\[
\Phi^{k}:H^{*}\left(B_{\mathrm{com}}G,\mathbb{R}\right)\rightarrow H^{*}\left(B_{\mathrm{com}}G,\mathbb{R}\right)
\]
when $k\in\mathbb{Z}$. By using characterizations of the cohomology
rings as well as the effect of these maps on them, we then use the
particularities of the action of the Weil group for $U(n)$, $SU(n)$
and $\mathrm{Sp}(n)$ to show that

\medskip{}

\begin{thm*}
For $G$ equal to $U\left(n\right)$, $SU\left(n\right)$ or $\mathrm{Sp}\left(n\right)$
then $H^{*}\left(B_{\mathrm{com}}G\right)$ is generated as an algebra
by 
\[
\left\{ \Phi^{k}\left(\mathrm{Im}\iota\right)\mid k\in\mathbb{Z}\setminus\left\{ 0\right\} \right\} .
\]
\end{thm*}
In order to compute the TC characteristic class associated to a generator
of the form $\Phi^{k}\left(\iota\left(s\right)\right)\in H^{*}\left(B_{\mathrm{com}}G,\mathbb{R}\right)$,
$s\in H^{*}\left(BG,\mathbb{R}\right)$ and $k\in\mathbb{Z}\setminus\left\{ 0\right\} $,
we develop another construction. For a TC structure $f:M\rightarrow B_{\mathrm{com}}G$
over a principal $G$-bundle $E\rightarrow M$, we construct a family
of principal $G$-bundles $E^{k}\rightarrow M$, called the $k$-th
associated bundles. Then we prove that if $\Omega_{k}$ is the curvature
of $E^{k}$ we have the equality 
\[
f^{*}\left(\Phi^{k}\left(\iota\left(s\right)\right)\right)=s\left(\Omega_{k}\right)\in H^{*}\left(M,\mathbb{R}\right),
\]
where $s\left(\Omega_{k}\right)$ is the characteristic class of $E^{k}\rightarrow M$
associated to $s$, which is computed using Chern-Weil theory. This
let us obtain our main result:

\medskip{}

\begin{thm*}
(Chern-Weil theory for TC structures) Consider $\varepsilon\in\left[M,B_{\mathrm{com}}G\right]$
an equivalence class with an underlying smooth vector bundle $E\rightarrow M$,
and structure group $U\left(n\right)$, $SU\left(n\right)$ or $\mathrm{Sp}\left(n\right)$.
Also let $\Omega_{k}$ be the curvature of $E^{k}$, the $k$-th associated
bundle of $E$. Then every TC characteristic class can be obtained
as a linear combinations of products of the form 
\[
s_{1}\left(\Omega_{k_{1}}\right)\cdot s_{2}\left(\Omega_{k_{2}}\right)\cdots s_{m}\left(\Omega_{k_{m}}\right)\in H^{*}\left(M,\mathbb{R}\right),
\]
where $s_{i}\in H^{*}\left(BG\right)$ and $k_{i}\in\mathbb{Z}$.
Each $s_{i}\left(\Omega_{k_{1}}\right)$ is the characteristic class
of the vector bundle $E^{k}\rightarrow M$ computed using its curvature. 
\end{thm*}
The outline of this article is as follows: in Section 2 we explain
the construction of $B_{\mathrm{com}}G$ and define TC structures,
TC characteristic classes, power maps and $k$-th associated bundles.
We also show how these concepts relate to each other. In Section 3
we show the effect of the power maps in the cohomology with real coefficients
of $B_{\mathrm{com}}G$. In Section 4 we obtain the generators for
$H^{*}\left(B_{\mathrm{com}}G,\mathbb{R}\right)$ for $G$ equal to
$U\left(n\right)$, $SU\left(n\right)$ or $\mathrm{Sp}\left(n\right)$.
In Section 5 we develop the Chern-Weil theory for TC characteristic
classes. Finally in Section 6 we show an example of a computation
of a TC characteristic class using a TC structure developed by Ramras
and Villareal (see \cite{key-15}).

\medskip{}

\section{TC structures and $B_{\mathrm{com}}G$}

In this section we introduce all the basic concepts we are using that
are related to commuting tuples in a Lie Group. 

\subsection{Simplicial construction for $B_{\mathrm{com}}G$: }

We first start with a basic description of the simplicial structure
used to define the space $B_{\mathrm{com}}G$. This will allow us
to obtain the generators for its cohomology with real coefficients. 

Let us define a simplicial space whose $n$-th level is given by $\mathrm{Hom}\left(\mathbb{Z}^{n},G\right)$,
which is the subspace of $G^{n}$ consisting of all commuting $n$-tuples.
This is $\left(g_{1},\ldots,g_{n}\right)$ such that $g_{i}g_{j}=g_{j}g_{i}$
for every $1\leq i,j\leq n$. Its face maps $\delta_{i}:\mathrm{Hom}\left(\mathbb{Z}^{n},G\right)\rightarrow\mathrm{Hom}\left(\mathbb{Z}^{n-1},G\right)$
are given by
\[
\delta_{i}\left(g_{1},\ldots,g_{n}\right):=\begin{cases}
\left(g_{2},\ldots,g_{n}\right) & i=0,\\
\left(g_{1},\ldots,,g_{i-1},g_{i}g_{i+1},g_{i+2},\ldots,g_{n}\right) & 1\leq i\leq n-1,\\
\left(g_{1},\ldots,g_{n-1}\right) & i=n,
\end{cases}
\]
and the degeneracy maps $s_{i}:\mathrm{Hom}\left(\mathbb{Z}^{n},G\right)\rightarrow\mathrm{Hom}\left(\mathbb{Z}^{n+1},G\right)$
are given by 
\[
s_{i}\left(g_{1},\ldots,g_{n}\right)=\left(g_{1},\ldots,,g_{i},1,g_{i+1},\ldots,g_{n}\right).
\]
It is routine to see that these maps satisfy the simplicial identities.

\medskip{}

\begin{defn}
The space $B_{\mathrm{com}}G$ is defined as the fat realization of
the simplicial space $\left\{ \mathrm{Hom}\left(\mathbb{Z}^{n},G\right)\right\} _{n\geq0}$,
that is 
\[
B_{com}G:=\left\Vert \mathrm{Hom}\left(\mathbb{Z}^{\bullet},G\right)\right\Vert .
\]
\end{defn}

It is also important to mention that for this construction the fat
realization is homotopy equivalent to the geometrical realization
as the simplicial space $\mathrm{Hom}\left(\mathbb{Z}^{\bullet},G\right)$
is proper. (See the appendix of \cite{key-1}.) 

In general, $\mathrm{Hom}\left(\mathbb{Z}^{m},G\right)$ is not connected.
The path connected component of $\mathrm{Hom}\left(\mathbb{Z}^{m},G\right)$
containing the element $\left(1,1,\ldots,1\right)$ is denoted by
$\mathrm{Hom}\left(\mathbb{Z}^{m},G\right)_{1}$. We can restrict
the face and degeneracy maps to obtain a simplicial space $\mathrm{Hom}\left(\mathbb{Z}^{\bullet},G\right)_{1}$
whose fat realization is denoted by $B_{\mathrm{com}}G_{1}$. However
Adem and Cohen showed in Corollary 2.4 of \cite{key-16} that $\mathrm{Hom}\left(\mathbb{Z}^{m},G\right)$
is path connected when $G$ is either $U\left(n\right)$, $SU\left(n\right)$
or $\mathrm{Sp}\left(n\right)$. So for these groups we have the equality
\[
B_{\mathrm{com}}G_{1}=B_{\mathrm{com}}G.
\]

\medskip{}

\subsection{Power Maps:}

For each $k\in\mathbb{Z}$ we define maps 
\begin{align*}
\Phi_{m}^{k}:\mathrm{Hom}\left(\mathbb{Z}^{m},G\right) & \rightarrow\mathrm{Hom}\left(\mathbb{Z}^{m},G\right)\\
\left(g_{1},\ldots,g_{m}\right) & \mapsto\left(g_{1}^{k},\ldots,g_{m}^{k}\right).
\end{align*}
These maps are well defined since the power of commuting elements
is still commutative. Commutativity is needed here in order for them
to induce simplicial maps. By this we mean maps commuting with the
face and degeneracy maps. More precisely we need the equality
\[
\left(g_{i}g_{i+1}\right)^{k}=g_{i}^{k}g_{i+1}^{k}
\]
to hold. Thus, only for commuting tuples we guarantee the existence
of the $k$-th power map $\Phi^{k}:B_{com}G\rightarrow B_{com}G$.
In the general bar construction for $G$, the power maps do not necessarily
induce simplicial maps. 

\medskip{}

\subsection{TC structures:}

Consider a principal $G$-bundle $\pi:E\rightarrow M$ over a compact
Hausdorff space $M$. This implies that $M$ has an open cover $\mathcal{U}:=\left\{ U_{i}\right\} _{i=1}^{m}$
and trivializations $\varphi_{i}:\pi^{-1}\left(U_{i}\right)\rightarrow U_{i}\times G.$
By considering the second component of the composition 
\[
\varphi_{j}\circ\varphi_{i}^{-1}:\left(U_{i}\cap U_{j}\right)\times G\rightarrow\left(U_{i}\cap U_{j}\right)\times G
\]
we obtain the cocycles $\rho_{ij}:U_{i}\cap U_{j}\rightarrow G$,
which satisfy that
\[
\varphi_{j}\circ\varphi_{i}^{-1}\left(x,g\right)=\left(x,\rho_{ij}\left(x\right)\cdot g\right),
\]
for every $x\in U_{i}\cap U_{j}$ and $g\in G$\footnote{It is worth recalling that up to equivalence the cocycles characterize
a principle bundle. }. Assume this cover is a good cover and consider the simplicial construction
of the nerve of the cover:
\[
\mathcal{N}\left(\mathcal{U}\right)_{n}=\bigsqcup\left(U_{i_{0}}\cap U_{i_{1}}\cdots\cap U_{i_{n}}\right).
\]
Take $\mathcal{N}\left(\mathcal{U}\right):=\left\Vert \mathcal{N}\left(\mathcal{U}\right)_{\bullet}\right\Vert $.
Since $\mathcal{U}$ is a good cover, $M$ and $\mathcal{N}\left(\mathcal{U}\right)$
are homotopy equivalent (See \cite{key-8}, Corollary 4G.3). This
guarantees a biyection 
\[
\left[\mathcal{N}\left(\mathcal{U}\right),Y\right]\cong\left[M,Y\right]
\]
for any space $Y$.

Recall that if we consider the bar construction of $BG$ we have in
every level the set of tuples, $G^{l}$. Then we have a simplicial
function $g_{n}:\mathcal{N}\left(\mathcal{U}\right)_{n}\rightarrow G^{n}$
given by
\[
g_{n}\left(x\right):=\left(\rho_{i_{0}i_{1}}\left(x\right),\rho_{i_{2}i_{3}}\left(x\right),\ldots,\rho_{i_{l-1}i_{l}}\left(x\right)\right).
\]
This induces a function $g:\mathcal{N}\left(\mathcal{U}\right)\rightarrow BG$,
which, up to homotopy, defines the classifying function $g:M\rightarrow BG$. 

Now suppose that for the principal $G$-bundle $\pi:E\rightarrow M$
there is a trivialization inducing cocycles that commute with each
other. That is that for $x\in U_{i}\cap U_{j}\cap U_{k}$ we have
\[
\rho_{ik}\left(x\right)\rho_{kj}\left(x\right)=\rho_{kj}\left(x\right)\rho_{ik}\left(x\right).
\]
Then we can define $f_{n}:\mathcal{N}\left(\mathcal{U}\right)_{n}\rightarrow\mathrm{Hom}\left(\mathbb{Z}^{n},G\right)$
given by
\[
f_{n}\left(x\right):=\left(\rho_{i_{0}i_{1}}\left(x\right),\rho_{i_{2}i_{3}}\left(x\right),\ldots,\rho_{i_{l-1}i_{l}}\left(x\right)\right).
\]
We have a commuting diagram 
\[
\xymatrix{\mathcal{N}\left(\mathcal{U}\right)_{n}\ar[r]^{f_{n}}\ar[rd]^{g_{n}} & \mathrm{Hom}\left(\mathbb{Z}^{n},G\right)\ar[d]\\
 & G^{n}
}
\]
where the vertical is the inclusion. This in turn leads to a diagram
commuting up to homotopy 
\[
\xymatrix{M\ar[r]^{f}\ar[rd]^{g} & B_{com}G\ar[d]\\
 & BG,
}
\]
where the vertical map is the inclusion. 

Adem and Gomez proved in Theorem 2.2 of \cite{key-1} that if there
is a lifting up to homotopy of the classifying function of a principal
$G$-bundle, then there exists a trivialization with commuting cocycles
for that bundle. That is, that the existence of a homotopy lifting
for the classifying function is a necessary and sufficient condition
for the existence of commuting cocycles for the principal $G$-bundle.
This allow us to define the following.

\medskip{}

\begin{defn}
Given a space $M$ and a principal $G$-bundle with classifying function
$f:M\rightarrow BG$, a \textbf{TC structure} over $M$ is a function
$g:M\rightarrow B_{\mathrm{com}}G$ such that 
\[
\xymatrix{M\ar[r]^{f}\ar[rd]^{g} & B_{com}G\ar[d]\\
 & BG,
}
\]
commutes up to homotopy. We say that two TC structures $f_{1},f_{2}:M\rightarrow B_{\mathrm{com}}G$
are equivalent if the functions are homotopic. 
\end{defn}

At this point is important to remark that given a principal bundle
there can be several different TC structures over it. That is, there
can exist functions $g:M\rightarrow BG$ and $f_{1},f_{2}:M\rightarrow B_{\mathrm{com}}G$
such that there are homotopies $\iota\circ f_{1}\cong g$ and $\iota\circ f_{2}\cong g$
but where $f_{1}$ is not homotopic to $f_{2}$. At the end of this
article we exhibit an example with a homotopy trivial $g:S^{4}\rightarrow BSU\left(2\right)$
with a non homotopy trivial lifting $G:S^{4}\rightarrow B_{\mathrm{com}}SU\left(2\right)$.

The assignment $\mathrm{Top}\rightarrow\left[-,B_{\mathrm{com}}G\right]$
is a contravariant functor: given a continuous function $h:M\rightarrow N$
we can consider its pullback 
\begin{align*}
h^{*}:\left[N,B_{\mathrm{com}}G\right] & \rightarrow\left[M,B_{\mathrm{com}}G\right]\\
\left[f\right] & \mapsto\left[f\circ h\right],
\end{align*}
where by $\left[f\right]$ we mean the homotopy class of the function
$f$. From this we define

\medskip{}

\begin{defn}
A characteristic class for TC structures or \textbf{TC characteristic
class} is a natural transformation $\eta:\left[-,B_{\mathrm{com}}G\right]\rightarrow H^{*}\left(-,\mathbb{R}\right)$.
Here $\left[-,B_{\mathrm{com}}G\right]$ is the functor assigning
to a space the set of homotopy classes of functions from the space
to $B_{\mathrm{com}}G$ and $H^{*}\left(-,\mathbb{R}\right)$ is the
functor of cohomology with real coefficients. 
\end{defn}

\medskip{}

\begin{prop}
There is a one to one correspondence between the TC characteristic
classes and elements of $H^{*}\left(B_{\mathrm{com}}G,\mathbb{R}\right)$.
\end{prop}

\begin{proof}
Consider a TC characteristic class $\eta$, and define 
\[
c_{\eta}:=\eta\left(B_{\mathrm{com}}G\right)\left(\left[\mathrm{Id}_{B_{\mathrm{com}}G}\right]\right)\in H^{*}\left(B_{\mathrm{com}}G,\mathbb{R}\right)
\]
 where $\mathrm{Id}_{B_{\mathrm{com}}G}$ is the identity on $B_{\mathrm{com}}G$.
We want to see that the assignment $\theta:\eta\mapsto c_{\eta}$
is a one to one and onto.

First, consider a continuous function $f:M\rightarrow B_{\mathrm{com}}G$.
By naturallity of $\eta$ we have a commuting diagram 

\[
\xymatrix{\left[B_{\mathrm{com}}G,B_{\mathrm{com}}G\right]\ar[rr]^{\eta\left(B_{\mathrm{com}}G\right)}\ar[d]_{f^{**}} &  & H^{*}\left(B_{\mathrm{com}}G,\mathbb{R}\right)\ar[d]^{f^{*}}\\
\left[M,B_{\mathrm{com}}G\right]\ar[rr]^{\eta\left(M\right)} &  & H^{*}\left(M,\mathbb{R}\right)
}
\]
where we use $f^{**}$to distinguish the pullback of the functor $\left[-,B_{\mathrm{com}}G\right]$
from the pullback from cohomology. The commutativity of the previous
diagram means that 
\[
f^{*}\left(c_{\eta}\right)=\eta\left(M\right)\left(f^{**}\left(\left[\mathrm{Id}_{B_{\mathrm{com}}G}\right]\right)\right).
\]
 But since 
\[
f^{**}\left(\left[\mathrm{Id}_{B_{\mathrm{com}}G}\right]\right)=\left[\mathrm{Id}_{B_{\mathrm{com}}G}\circ f\right]=\left[f\right],
\]
we can conclude that 
\begin{equation}
f^{*}\left(c_{\eta}\right)=\eta\left(M\right)\left(\left[f\right]\right)\in H^{*}\left(M,\mathbb{R}\right).\label{eq:TC classes and classes of cohom BcomG}
\end{equation}

Now, to see that $\theta:\eta\mapsto c_{\eta}$ is surjective, take
$c\in H^{*}\left(B_{\mathrm{com}}G.\mathbb{R}\right)$ and define
\begin{align*}
\eta_{c}\left(M\right):\left[M,B_{\mathrm{com}}G\right] & \rightarrow H^{*}\left(M,\mathbb{R}\right)\\
\left[f\right] & \mapsto f^{*}\left(c\right).
\end{align*}
This can be seen to be a well defined natural transformation thanks
to the properties of cohomology, so $\eta_{c}$ is a TC characteristic
class. Now, by definition and Equation \ref{eq:TC classes and classes of cohom BcomG}
it follows that 
\[
c_{\eta_{c}}=\eta_{c}\left(B_{\mathrm{com}}G\right)\left(\left[\mathrm{Id}_{B_{\mathrm{com}}G}\right]\right)=\mathrm{Id}_{B_{\mathrm{com}}G}^{*}\left(c\right)=c,
\]
which implies that $\theta:\eta\mapsto c_{\eta}$ sends $\eta_{c}$
into $c$. 

On the other hand to prove inyectivity, take $c_{\eta}$ and consider
$\eta_{c_{\eta}}$ as defined before. For any $f:M\rightarrow B_{\mathrm{com}}G$
it follows that
\[
\eta_{c_{\eta}}\left(M\right)\left(\left[f\right]\right)=f^{*}\left(c_{\eta}\right)=\eta\left(M\right)\left(\left[f\right]\right),
\]
where the first equality is true by definition, and the sencond thanks
to Equation \ref{eq:TC classes and classes of cohom BcomG}. The equality
$\eta_{c_{\eta}}\left(M\right)\left(\left[f\right]\right)=\eta\left(M\right)\left(\left[f\right]\right)$
means that $\eta_{c_{\eta}}=\eta$, so if $\omega$ is another TC
characteristic class such that $c_{\eta}=c_{\omega}$, then 
\[
n=\eta_{c_{\eta}}=\eta_{c_{\omega}}=\omega.
\]
That is, $\theta:\eta\mapsto c_{\eta}$ is inyective. 
\end{proof}
\medskip{}

\subsection{The $k$-th associated bundles: }

Once again let $\left\{ U_{\alpha}\right\} _{\alpha\in J}$ be an
open cover of a space $M$ such that there are trivializations with
a cocycle $\left\{ \rho_{ij}:U_{i}\cap U_{j}\rightarrow G\right\} $,
such that if $x\in U_{i}\cap U_{j}\cap U_{l}$ then 
\[
\rho_{il}\left(x\right)\rho_{lj}\left(x\right)=\rho_{lj}\left(x\right)\rho_{lk}\left(x\right).
\]
These transition functions satisfy the cocycle condition as well,
that is, 
\[
\rho_{ij}\left(x\right)=\rho_{il}\left(x\right)\rho_{lj}\left(x\right).
\]
In particular these two properties imply that for $k\in\mathbb{Z}$
we have 
\[
\rho_{ij}\left(x\right)^{k}=\left(\rho_{il}\left(x\right)\rho_{lj}\left(x\right)\right)^{k}=\rho_{il}\left(x\right)^{k}\rho_{lj}\left(x\right)^{k}.
\]
This tell us that the collection of functions $\rho_{ij}^{k}:U_{i}\cap U_{j}\rightarrow G$
defined as 
\[
\rho_{ij}^{k}\left(x\right):=\rho_{ij}\left(x\right)^{k}
\]
 also satisfy the cocycle condition. Then for each $k\in\mathbb{Z}$
we can construct a principal bundle $p_{k}:E^{k}\rightarrow M$ with
trivializations over $\left\{ U_{i}\right\} _{i\in I}$ with cocycle
$\left\{ \rho_{ij}^{k}\right\} $. We call it the \textbf{$k$-th
associated bundle} of $E$. Here $E^{k}$ is obtained as the quotient
space 
\[
\left(\bigsqcup_{i\in I}U_{i}\times G\right)\diagup\sim,
\]
where for $x\in U_{i}$ and $y\in U_{j}$, $\left(x,g\right)\sim\left(y,h\right)$
if only if $x=y$ and $\rho_{ij}^{k}\left(x\right)\cdot g=h$.

\medskip{}

\begin{prop}
\label{thm:(Classifying-functions-for assiciated bundles}(Classifying
functions for $k$- th associated bundles) If $f:M\rightarrow B_{com}G$
defines a TC structure over a principal $G$-bundle, and $f^{k}:M\rightarrow B_{com}G$
is the corresponding lifting over the $k$-th associated bundle, then
the following map diagram commutes
\begin{equation}
\xymatrix{M\ar[r]^{f}\ar[rd]_{f^{k}} & B_{com}G\ar[d]^{\Phi^{k}}\\
 & B_{com}G.
}
\label{eq:associated bundle class func.}
\end{equation}
Where $\Phi^{k}:B_{com}G\rightarrow B_{com}G$ are the power maps. 
\end{prop}

\medskip{}

\begin{proof}
As it was explained before, to obtain the classifying functions for
the $k$-th associated bundle $p\left(k\right):E^{k}\rightarrow M$
we need to consider a simplicial map $f_{l}^{k}:\mathcal{N}\left(\mathcal{U}\right)_{l}\rightarrow\mathrm{Hom}\left(\mathbb{Z}^{l},G\right)$.
The components of this function are given by the transition functions:
if $x\in U_{i_{1}}\cap U_{i_{2}}\cap\cdots\cap U_{i_{l+1}}$, we take
\[
f_{l}^{k}\left(x\right)=\left(\rho_{i_{0}i_{1}}^{k}\left(x\right),\ldots,\rho_{i_{l-1}i_{l}}^{k}\left(x\right)\right)=\left(\rho_{i_{0}i_{1}}\left(x\right)^{k},\ldots,\rho_{i_{l-1}i_{l}}\left(x\right)^{k}\right).
\]
This can be rewritten using the power functions as
\[
f_{l}^{k}=\Phi_{l}^{k}\circ f_{l}.
\]
The desired result is obtained after passing to the geometric realization. 

\medskip{}
\end{proof}

\section{Power maps and cohomology of $B_{com}G$}

In this section we go through the reasoning behind the computation
of 
\[
H^{*}\left(B_{com}G_{1},\mathbb{R}\right)
\]
 made in \cite{key-1} to obtain the effect of the power maps on cohomology.
Thus we track such effect in the main steps of the computation. To
make notation simpler, we assume $B_{com}G_{1}=B_{com}G$, which is
true when $G$ is either $U\left(n\right)$, $SU\left(n\right)$ or
$\mathrm{Sp}\left(n\right)$, as mentioned before. We also fix a maximal
torus $T\subseteq G$ with Weyl group $W$ and we write $H^{*}\left(Y\right)$
to refer to the cohomology of $Y$ with real coefficients.

In Section 7 of \cite{key-1} it is proved that for a maximal torus
$T$ of $G$ with Weyl group $W$ we have 
\[
H^{*}\left(B_{com}G\right)\cong\left(H^{*}\left(BT\right)\otimes H^{*}\left(BT\right)\right)^{W}/J,
\]
where $J$ is the ideal generated by the see
\[
\left\{ f\left(x\right)\otimes1\in H^{*}\left(BT\right)\otimes H^{*}\left(BT\right)\mid f\mathrm{\:of\:positive\:degree}\right.
\]
\[
\mathrm{\left.\mathrm{polynomial}\:\mathrm{and}\:n\cdot f\left(x\right)=f\left(x\right)\:\mathrm{for}\:\mathrm{all}\:n\in W\right\} }
\]
and $W$ acts on $H^{*}\left(BT\right)\otimes H^{*}\left(BT\right)$
diagonally. 

In order to reach the description of the induced power maps $\Phi^{k}$
on cohomology, we need to consider some auxiliary maps that are used
in \cite{key-1}. In this process we will see what is their relationship
with the power maps. First, since all the tuples of $T^{m}$ have
commuting elements, we can consider the power maps for the torus $\psi^{k}:H^{*}\left(BT\right)\rightarrow H^{*}\left(BT\right)$.
This is the map induced in the $m$-level the by the power maps, 
\begin{align*}
\Phi_{m}^{k}:\mathrm{Hom}\left(\mathbb{Z}^{m},T\right) & \rightarrow\mathrm{Hom}\left(\mathbb{Z}^{m},T\right)\\
\left(g_{1},\ldots,g_{m}\right) & \mapsto\left(g_{1}^{k},\ldots,g_{m}^{k}\right).
\end{align*}
Also consider
\begin{align*}
\varphi_{m}:G\times T^{m} & \rightarrow\mathrm{Hom}\left(\mathbb{Z}^{m},G\right)\\
\left(g,t_{1},\ldots,t_{n}\right) & \mapsto\left(gt_{1}g^{-1},\ldots,gt_{m}g^{-1}\right).
\end{align*}
Because $\mathrm{Hom}\left(\mathbb{Z}^{m},G\right)$ is path connected,
an $m$-tuple $\left(g_{1},\ldots,g_{m}\right)$ has commuting elements
if and only if there is a maximal tori containing all $g_{i}$ (see
Lemma 4.2 of \cite{key-9}). Then, since every maximal tori is conjugated
to $T$, the previous map is surjective. We also have an action of
the normalizer of $T$ in $G$, $N_{G}\left(T\right)$, on $G\times T^{m}$,
where for $\eta\in N_{G}\left(T\right)$ we have
\[
\eta\cdot\left(g,t_{1},\ldots,t_{m}\right)=\left(g\eta^{-1},\eta t_{1}\eta^{-1},\ldots,\eta t_{m}\eta^{-1}\right).
\]
On the other hand, consider the flag variety $G/T$. It is easy to
verify that the maps $\varphi_{m}$ factor through the product $G/T\times T^{m}$
giving us a commutative diagram
\[
\xymatrix{G\times T^{m}\ar[r]^{\varphi_{m}}\ar[d] & \mathrm{Hom}\left(\mathbb{Z}^{m},G\right)\\
G/T\times T^{m}\ar[ru]
}
,
\]
such that the diagonal map is also surjective. We call it $\varphi_{m}$
as well. This family of maps give rise to a simplicial map 
\[
\varphi_{\bullet}:G/T_{\bullet}\times T^{\bullet}\rightarrow\mathrm{Hom}\left(\mathbb{Z}^{\bullet},G\right).
\]
Here $G/T_{\bullet}$ is the trivial simplicial space with $G/T$
on every level, and $T^{\bullet}$ is the simplicial space obtained
by the bar construction for the classifying space applied to $T$.

Furthermore using representatives of the Weyl group $\left[\eta\right]\in W=N_{G}\left(T\right)/T$,
we have a well defined action on $G/T\times T^{m}$ given by
\[
\left[\eta\right]\cdot\left(\left[g\right],t_{1},\ldots,t_{m}\right)=\left(\left[g\eta^{-1}\right],\eta t_{1}\eta^{-1},\ldots,\eta t_{m}\eta^{-1}\right).
\]
It is easy to see that this action makes $\varphi_{m}$ $W$-invariant.
Also we can construct a simplicial space, $G/T\times_{W}T^{\bullet}$,
having the space of orbits $G/T\times_{W}T^{m}$ on the $m$-th level.
Where the simplicial structure is inhered form $G/T_{\bullet}\times T^{\bullet}$,
giving us a simplicial map $\pi_{\bullet}:G/T_{\bullet}\times T^{\bullet}\rightarrow G/T\times_{W}T^{\bullet}$
where on each level we have the natural quotient map. Then we have
a commuting diagram
\[
\xymatrix{G/T_{\bullet}\times T^{\bullet}\ar[r]^{\varphi_{\bullet}}\ar[d]^{\pi_{\bullet}} & \mathrm{Hom}\left(\mathbb{Z}^{\bullet},G\right),\\
G/T\times_{W}T^{\bullet}\ar[ru]^{\bar{\varphi}_{\bullet}}
}
\]
where $\bar{\varphi}_{m}:G/T\times_{W}T^{m}\rightarrow\mathrm{Hom}\left(\mathbb{Z}^{m},G\right)$
is the induced map. Finally, we have maps 
\begin{align*}
P_{m}^{k}:G/T\times T^{m} & \rightarrow G/T\times T^{m}\\
\left(\left[g\right],t_{1},\ldots,t_{m}\right) & \mapsto\left(\left[g\right],t_{1}^{k},\ldots,t_{m}^{k}\right).
\end{align*}
By direct computation it can be seen that these maps are compatible
with the simplicial structure. They are also $W$-equivariant, that
is
\[
\left[\eta\right]\cdot P_{m}^{k}\left(g,t_{1},\ldots,t_{m}\right)=P_{m}^{k}\left(\left[\eta\right]\cdot\left(g,t_{1},\ldots,t_{m}\right)\right),
\]
This is true since, $\left(\eta t\eta^{-1}\right)^{k}=\eta t^{k}\eta^{-1}$
for $t\in T$.

From this point we are showing that the arguments given in \cite{key-1}
are natural. This allow us to include the power maps in their conclusions. 

\medskip{}

\begin{prop}
\label{prop:Isomorphism on H n'th level. }Suppose $G$ is a compact
connected Lie group such that 
\[
\mathrm{Hom}\left(\mathbb{Z}^{m},G\right)
\]
 is path connected for every non negative integer $m$. Then for the
cohomology with real coefficients we have a commutative diagram
\begin{equation}
\xymatrix{H^{*}\left(\mathrm{Hom}\left(\mathbb{Z}^{m},G\right)\right)\ar[r]^{\varphi_{m}^{*}}\ar[d]_{\left(\Phi_{m}^{k}\right)^{*}} & H^{*}\left(G/T\times T^{m}\right)^{W}\ar[d]^{\left(P_{m}^{k}\right)^{*}}\\
H^{*}\left(\mathrm{Hom}\left(\mathbb{Z}^{m},G\right)\right)\ar[r]^{\varphi_{m}^{*}} & H^{*}\left(G/T\times T^{m}\right)^{W}.
}
\label{eq: Isomorphism on H n-th level.}
\end{equation}
where the horizontal maps are isomorphisms. 
\end{prop}

\begin{proof}
Under this setting, Theorem 3.3 of \cite{key-9} is applied to conclude
that we have the following natural isomorphisms
\begin{equation}
H^{*}\left(\mathrm{Hom}\left(\mathbb{Z}^{m},G\right)\right)\overset{\left(\bar{\varphi}_{m}\right)^{*}}{\cong}H^{*}\left(G/T\times_{W}T^{m}\right)\overset{\pi^{*}}{\cong}H^{*}\left(G/T\times T^{m}\right)^{W}.\label{eq:Ident. coho Bcom}
\end{equation}
 Now let us see how the power maps are related to this constructions
so far. We have maps 
\begin{align*}
P_{m}^{k}:G/T\times T^{m} & \rightarrow G/T\times T^{m}\\
\left(\left[g\right],t_{1},\ldots,t_{m}\right) & \mapsto\left(\left[g\right],t_{1}^{k},\ldots,t_{m}^{k}\right).
\end{align*}
By direct computation it can be seen that these maps are compatible
with the simplicial structure. They are also $W$-equivariant, that
is
\[
\left[\eta\right]\cdot P_{m}^{k}\left(g,t_{1},\ldots,t_{m}\right)=P_{m}^{k}\left(\left[\eta\right]\cdot\left(g,t_{1},\ldots,t_{m}\right)\right),
\]
This is true since, $\left(\eta t\eta^{-1}\right)^{k}=\eta t^{k}\eta^{-1}$
for $t\in T$. Thus, they induced a well define map $\bar{P}_{m}^{k}:G/T\times_{W}T^{m}\rightarrow G/T\times_{W}T^{m}$,
and we get the following commuting diagram
\[
\xymatrix{H^{*}\left(G/T\times_{W}T^{m}\right)\ar[r]^{\pi^{*}}\ar[d]^{\left(\bar{P}_{m}^{k}\right)^{*}} & H^{*}\left(G/T\times T^{m}\right)\ar[d]^{P^{k}}\\
H^{*}\left(G/T\times_{W}T^{m}\right)\ar[r]^{\pi^{*}} & H^{*}\left(G/T\times T^{m}\right).
}
\]
We know that the homomorphism 
\[
\pi^{*}:H^{*}\left(G/T\times_{W}T^{m}\right)\rightarrow H^{*}\left(G/T\times T^{m}\right)
\]
 actually has image equal to $H^{*}\left(G/T\times T^{m}\right)^{W}$,
since 
\[
H^{*}\left(G/T\times_{W}T^{m}\right)\overset{\pi^{*}}{\cong}H^{*}\left(G/T\times T^{m}\right)^{W}.
\]
 Thus, we actually have the diagram
\[
\xymatrix{H^{*}\left(G/T\times_{W}T^{m}\right)\ar[r]^{\pi^{*}}\ar[d]^{\left(\bar{P}_{m}^{k}\right)^{*}} & H^{*}\left(G/T\times T^{m}\right)^{W}\ar[d]^{P^{k}}\\
H^{*}\left(G/T\times_{W}T^{m}\right)\ar[r]^{\pi^{*}} & H^{*}\left(G/T\times T_{m}\right)^{W}.
}
\]
where the horizontal maps are isomorphism. This implies that $\left(P_{m}^{k}\right)^{*}$
preserves $W$-invariance: 
\[
\left(P_{m}^{k}\right)^{*}\left(H^{*}\left(G/T\times T^{m}\right)^{W}\right)\subseteq H^{*}\left(G/T\times T^{m}\right)^{W}.
\]

Also, by direct computation from the definitions and the fact that
$\left(gtg^{-1}\right)^{k}=gt^{k}g^{-1}$, it follows that $\varphi_{m}\circ P_{m}^{k}=\Phi_{m}^{k}\circ\varphi_{m}$
holds. And since $\left(P_{m}^{k}\right)^{*}$ preserves $W$-invariance,
we obtain the following commuting diagram
\[
\xymatrix{H^{*}\left(\mathrm{Hom}\left(\mathbb{Z}^{m},G\right)\right)\ar[r]^{\varphi_{m}^{*}}\ar[d]_{\left(\Phi_{m}^{k}\right)^{*}} & H^{*}\left(G/T\times T^{m}\right)^{W}\ar[d]^{\left(P_{m}^{k}\right)^{*}}\\
H^{*}\left(\mathrm{Hom}\left(\mathbb{Z}^{m},G\right)\right)\ar[r]^{\varphi_{m}^{*}} & H^{*}\left(G/T\times T^{m}\right)^{W}.
}
\]
Here the horizontal maps are isomorphism as they can be factored by
the isomorphisms 
\[
\bar{\varphi}_{m}^{*}:H^{*}\left(\mathrm{Hom}\left(\mathbb{Z}^{m},G\right)\right)\rightarrow H^{*}\left(G/T\times_{W}T^{m}\right)
\]
 and 
\[
\pi^{*}:H^{*}\left(G/T\times_{W}T^{m}\right)\rightarrow H^{*}\left(G/T\times T^{m}\right)^{W}.
\]
\end{proof}
Next, thanks to the naturality in Theorems 5.15 and 1.19 of \cite{key-10}
it can be concluded that 
\begin{prop}
\label{prop:Cohomology of simplicial spaces}Let $X_{\bullet}$ and
$Y_{\bullet}$ be two simplicial spaces with a simplicial map $f:X_{\bullet}\rightarrow Y_{\bullet}$.
Suppose also that there is a finite group $K$ with an action on every
level $X_{q}$ compatible with the simplicial structure, such that
there is an isomorphism $H^{p}\left(C^{*}\left(X_{q}\right)\right)^{K}\cong H^{p}\left(C^{*}\left(Y_{q}\right)\right)$
induce on every level by the maps of $f$. Then there is natural isomorphism
\[
\left\Vert f\right\Vert ^{*}:H^{*}\left(\left\Vert Y\right\Vert \right)\rightarrow H^{*}\left(\left\Vert X\right\Vert \right)^{K},
\]
where $\left\Vert X\right\Vert $ and $\left\Vert Y\right\Vert $
are the fat realizations. 
\end{prop}

\medskip{}

\begin{prop}
\label{prop:BcomG and the realizations}Suppose $G$ is a compact
connected Lie group such that 
\[
\mathrm{Hom}\left(\mathbb{Z}^{m},G\right)
\]
 is path connected for every non negative integer $m$. Then for the
cohomology with real coefficients we have a commutative diagram
\begin{equation}
\xymatrix{H^{*}\left(B_{\mathrm{com}}G\right)\ar[rrr]^{\pi^{*}}\ar[d]_{\Phi^{k}} &  &  & H^{*}\left(\left\Vert \left(G/T\right)_{\bullet}\times BT_{\bullet}\right\Vert \right)^{W}\ar[d]^{P^{k}}\\
H^{*}\left(B_{\mathrm{com}}G\right)\ar[rrr]^{\pi^{*}} &  &  & H^{*}\left(\left\Vert \left(G/T\right)_{\bullet}\times BT_{\bullet}\right\Vert \right)^{W},
}
\label{eq:Diagram for invaiant cohomoloty-1-1}
\end{equation}
where the horizontal maps are isomorphisms. Here we are abusing notation
by using the same names for the power map and its induce map on cohomology.
\end{prop}

\begin{proof}
Because of Proposition \ref{prop:Isomorphism on H n'th level. } the
conditions of Proposition \ref{prop:Cohomology of simplicial spaces}
can be applied to conclude that $\pi^{*}:H^{*}\left(B_{\mathrm{com}}G\right)\rightarrow H^{*}\left(\left\Vert \left(G/T\right)_{\bullet}\times BT_{\bullet}\right\Vert \right)^{W}$
is an isomorphism. Then Diagram \ref{eq: Isomorphism on H n-th level.}
implies that Diagram \ref{eq:Diagram for invaiant cohomoloty-1-1}
commutes. 
\end{proof}
\medskip{}

\begin{prop}
\label{thm: Cohomology and Power maps} Suppose $G$ is a compact
connected Lie group such that 
\[
\mathrm{Hom}\left(\mathbb{Z}^{m},G\right)
\]
 is path connected for every non negative integer $m$. Then for the
cohomology with real coefficients we have a commutative diagram
\[
\xymatrix{H^{*}\left(B_{com}G\right)\ar[r]\ar[d]^{\Phi^{k}} & \left(H^{*}\left(BT\right)\otimes H^{*}\left(BT\right)\right)^{W}/J\ar[d]^{\overline{\mathrm{Id}\otimes\psi^{k}}}\\
H^{*}\left(B_{com}G\right)\ar[r] & \left(H^{*}\left(BT\right)\otimes H^{*}\left(BT\right)\right)^{W}/J,
}
\]
where the horizontal maps are the same isomorphism given above, and
$\Phi^{k}$ are the power maps on cohomology. 
\end{prop}

\begin{proof}
In the Diagram \ref{eq:Diagram for invaiant cohomoloty-1-1} we can
consider the naturality on the Kunneth formulas and the fact that
the realization of the simplicial product are naturally isomorphic
to the product of the realizations of each of the simplicial spaces
involved (see Theorem 14.3 of \cite{key-6}). Then we obtain the commuting
diagram 
\begin{equation}
\xymatrix{H^{*}\left(B_{com}G\right)\ar[r]\ar[d]^{\Phi^{k}} & \left(H^{*}\left(G/T\right)\otimes H^{*}\left(BT\right)\right)^{W}\ar[d]^{\mathrm{Id}\otimes\psi^{k}}\\
H^{*}\left(B_{com}G\right)\ar[r] & \left(H^{*}\left(G/T\right)\otimes H^{*}\left(BT\right)\right)^{W},
}
\label{eq:Power maps diagram 1}
\end{equation}
where the horizontal maps are still isomorphisms. 

Next in the proof of Proposition 7.1 of \cite{key-1} they replace
$H^{*}\left(G/T\right)$ giving natural arguments, which allow us
to obtain our conclusion. 
\end{proof}
The last result is important since it tell us that in order to obtain
the effect of power maps on cohomology of $B_{\mathrm{com}}G$\footnote{When $\mathrm{Hom}\left(\mathbb{Z}^{m},G\right)$ is path connected
for every $m$.}, we need to understand their effect when the Lie group is a torus,
$T=\left(S^{1}\right)^{n}$. 

\medskip{}

\begin{thm}
\label{thm: Power maps for torus}Consider the $k$-th power map 
\begin{align*}
\psi^{k}:T & \rightarrow T\\
\left(t_{1},\ldots,t_{n}\right) & \mapsto\left(t_{1}^{k},\ldots,t_{n}^{k}\right).
\end{align*}
 Then by identifying $H^{*}\left(BT\right)\cong\mathbb{R}\left[x_{1},\ldots,x_{n}\right]$,
the induced $k$-th power map is characterized by 
\begin{align*}
\mathbb{R}\left[x_{1},\ldots,x_{n}\right] & \rightarrow\mathbb{R}\left[x_{1},\ldots,x_{n}\right]\\
x_{i} & \mapsto kx_{i}.
\end{align*}
\end{thm}

\begin{proof}
On a circle the $k$-th power of its elements induces the multiplication
by $k$ on the fundamental group: if 
\[
S^{1}=\left\{ z\in\mathbb{C}\mid\left|z\right|=1\right\} 
\]
then the $k$-th power map is given by 
\begin{align*}
\eta:S^{1} & \rightarrow S^{1}\\
z & \mapsto z^{k}
\end{align*}
which is know to be a map of degree $k$. This means that if identify
$\pi_{1}\left(S^{1}\right)\cong\mathbb{Z}$ then the $k$-th power
maps induces multiplication by $k$ on the fundamental group. 

Consider the projections
\begin{align*}
p_{i}:\left(S^{1}\right)^{n} & \rightarrow S^{1}\\
\left(z_{1},\ldots,z_{n}\right) & \mapsto z_{i}.
\end{align*}
It is well known that the map $q:\pi_{1}\left(\left(S^{1}\right)^{n}\right)\rightarrow\pi_{1}\left(S^{1}\right)^{n}$
given by 
\[
q\left(\left[\alpha\right]\right):=\left(\left[p_{1}\circ\alpha\right],\ldots,\left[p_{n}\circ\alpha\right]\right)
\]
is an isomoprihsm, since $S^{1}$ is path connected. Since the power
map $\psi^{k}:T\rightarrow T$ considers the $k$-th power component
wise, it follows that we have a commutative diagram
\[
\xymatrix{\pi_{1}\left(T\right)\ar[r]^{q}\ar[d]^{\psi_{*}^{k}} & \pi_{1}\left(S^{1}\right)^{n}\ar[d]^{\prod\eta_{*}}\\
\pi_{1}\left(T\right)\ar[r]^{q} & \pi_{1}\left(S^{1}\right)^{n}.
}
\]
Since the horizontal maps are isomorphisms, we have the characterization
\begin{align*}
\psi_{*}^{k}:\pi_{1}\left(T\right) & \rightarrow\pi_{1}\left(T\right)\\
\alpha & \mapsto k\alpha
\end{align*}
where we see $\alpha=\left(\alpha_{1},\ldots,\alpha_{n}\right)\in\mathbb{Z}^{n}$,
and $k\alpha=\left(k\alpha_{1},\ldots,k\alpha_{n}\right)$. 

Now let us consider the fiber sequence of the classifying space of
the torus 
\[
T\rightarrow ET\rightarrow BT.
\]
This induces an exact sequence on homotopy groups 
\[
\cdots\rightarrow\pi_{m}\left(ET\right)\rightarrow\pi_{m}\left(BT\right)\rightarrow\pi_{m-1}\left(T\right)\rightarrow\cdots\pi_{1}\left(ET\right)\rightarrow\pi_{1}\left(BT\right)\rightarrow
\]
\[
\rightarrow\pi_{0}\left(T\right)\rightarrow\pi_{0}\left(ET\right)\rightarrow\pi_{0}\left(BT\right)\rightarrow0
\]
but since $ET$ is contractible, we get an isomorphism $\pi_{m}\left(BT\right)\rightarrow\pi_{m-1}\left(T\right)$.
In particular we get 
\[
\pi_{m}\left(BT\right)\text{=}\begin{cases}
\mathbb{Z}^{n} & m=2,\\
0 & \mathrm{otherwise}.
\end{cases}
\]
Since the exact sequence is natural, we get that the power map on
$BT$ induces the multiplication by $k$ on the second homotopy group.
Furthermore since $BT$ is simply connected, by Hurewicz\textasciiacute s
theorem we get that $H_{2}\left(BT,\mathbb{Z}\right)\cong\pi_{2}\left(BT\right)$,
and once again because of naturality the effect on the second homology
is multiplication by $k$. 

We now apply the universal coefficients theorem to get that 
\[
H^{2}\left(BT\right)\cong\mathrm{Hom}\left(H_{2}\left(BT,\mathbb{Z}\right),\mathbb{R}\right)\cong\mathbb{R}^{n}.
\]
Naturality allow us to conclude that the effect of the $k$-th power
map is once again multiplication by $k$. Finally it is known that
the real cohomology of $BT$ is the polynomial ring $\mathbb{R}\left[x_{1},\ldots,x_{n}\right]$
where $x_{i}\in H^{2}\left(BT\right)$ for $1\leq i\leq n$ (see \cite{key-10},
Proposition 8.11). Since we know that the effect of the $k$ power
map is multiplication by $k$ on the $x_{i}$, this determines the
effect on the whole cohomology ring. 
\end{proof}
As corollary of Proposition \ref{thm: Cohomology and Power maps}
and Theorem \ref{thm: Power maps for torus} we obtain the following:

\medskip{}

\begin{thm}
Identify the real cohomology ring of an $n$-torus with $\mathbb{R}\left[x_{1},\ldots,x_{n}\right]$,
and suppose that $G$ is a Lie group such that 
\[
\mathrm{Hom}\left(\mathbb{Z}^{m},G\right)=\mathrm{Hom}\left(\mathbb{Z}^{m},G\right)_{1}
\]
 for every $m$, then 
\[
H^{\text{*}}\left(B_{com}G\right)\cong\left(\mathbb{R}\left[x_{1},\ldots,x_{n}\right]\otimes\mathbb{R}\left[y_{1},\ldots,y_{n}\right]\right)^{W}/J.
\]
Here $J$ is the ideal generated by the invariant polynomials of positive
degree on the $x_{i}$ under the action of the Weyl group, $W$. Further,
the power maps $\Phi^{k}:H^{\text{*}}\left(B_{com}G,F\right)\rightarrow H^{\text{*}}\left(B_{com}G,F\right)$
are induced by the homomorphism characterized by sending $x_{i}\mapsto x_{i}$
and $y_{i}\mapsto ky_{i}$ for every $1\leq i\leq n$.
\end{thm}

\medskip{}

\section{Generators of $H^{*}\left(B_{com}G,\mathbb{R}\right)$ for $G=U\left(n\right),\mathrm{Sp}\left(n\right)\:\mathrm{and}\:SU\left(n\right)$}

Here we are going to examine the cases of the Lie groups 
\[
G=U\left(n\right),\mathrm{Sp}\left(n\right)\:\mathrm{and}\:SU\left(n\right).
\]
 From this we see that possible differences between the different
cases depend entirely on the Weyl group and its action on the cohomology
of $BT$. For this we first need to establish some facts and definitions. 

If $B_{\mathrm{com}}G_{1}=B_{\mathrm{com}}G$ we know that
\[
H^{\text{*}}\left(B_{com}G,\mathbb{R}\right)\cong\left(\mathbb{R}\left[x_{1},\ldots,x_{n}\right]\otimes\mathbb{R}\left[y_{1},\ldots,y_{n}\right]\right)^{W}/J.
\]
While in general it is known that for a compact and connected Lie
group $G$
\[
H^{\text{*}}\left(BG,\mathbb{R}\right)\cong H^{\text{*}}\left(BT,\mathbb{R}\right)^{W}\cong P\left[\mathfrak{t}\right]^{W},
\]
where $W$ is its Weil group. Its action is induced by adjunction.
That is, if $\mathfrak{t}$ is the Lie algebra of $T$, $P\left[\mathfrak{t}\right]$
is the polynomial algebra of $\mathfrak{t}$. An element $\left[n\right]\in W\cong N_{G}\left(T\right)/T$
has a well defined action given by adjunction, $\mathrm{ad}\left(n\right):\mathfrak{t}\rightarrow\mathfrak{t}$.
This in turn induces an action of $W$ on $P\left[\mathfrak{t}\right]$.
Even further, if $n$ is the dimension of $\mathfrak{t}$, $P\left[\mathfrak{t}\right]$
can be identified with $\mathbb{R}\left[z_{1},\ldots,z_{n}\right]$,
and under such identification, we have an action of $W$ on the latter. 

There is a natural inclusion $B_{\mathrm{com}}G\hookrightarrow BG$
, inducing a map 
\[
\iota:H^{\text{*}}\left(BG,\mathbb{R}\right)\rightarrow H^{\text{*}}\left(B_{com}G,\mathbb{R}\right).
\]
In terms of the previous identifications, $\iota$ is induced by the
homomorphism (\cite{key-14}, Corollary A.2.)
\begin{align*}
\mathbb{R}\left[z_{1},\ldots,z_{n}\right] & \rightarrow\mathbb{R}\left[x_{1},\ldots,x_{n}\right]\otimes\mathbb{R}\left[y_{1},\ldots,y_{n}\right]\\
z_{i} & \mapsto x_{i}+y_{i}.
\end{align*}
Additionally we saw in the previous section that the power maps, 
\[
\Phi^{k}:H^{\text{*}}\left(B_{com}G,\mathbb{R}\right)\rightarrow H^{\text{*}}\left(B_{com}G,\mathbb{R}\right)
\]
 are induced by the map characterized by sending $x_{i}\mapsto x_{i}$
and $y_{i}\mapsto ky_{i}$ for every $1\leq i\leq n$. 
\begin{defn}
We call the subalgebra generated by $\left\{ \Phi^{k}\left(\mathrm{Im}\iota\right)\mid k\in\mathbb{Z}\setminus\left\{ 0\right\} \right\} \subset H^{\text{*}}\left(B_{com}G,\mathbb{R}\right)$
by
\[
\mathcal{S}:=\left\langle \Phi^{k}\left(\mathrm{Im}\iota\right)\mid k\in\mathbb{Z}\setminus\left\{ 0\right\} \right\rangle .
\]
On this section we use the previous maps to see that if $G=U\left(n\right),\mathrm{Sp}\left(n\right)\:\mathrm{and}\:SU\left(n\right)$
then $\mathcal{S}$ is all of $H^{*}\left(B_{\mathrm{com}}G,\mathbb{R}\right)$.
We do this by dealing with the explicit descriptions of their actions
and the specific Weyl groups on each case. 
\end{defn}

Before dealing with each individual case, it is worth proving the
following
\begin{lem}
The subalgebra $\mathcal{S}$ is closed under the power maps.
\end{lem}

\begin{proof}
This is true since $\Phi^{k}$ is a $\mathbb{R}$-homomorphism of
algebras as well because of the equality $\Phi^{k}\circ\Phi^{l}=\Phi^{kl}$.
This implies that for $q_{j}\in\mathbb{R}\left[z_{1},\ldots,z_{n}\right]$,
$\alpha_{j}\in\mathbb{R}$
\[
\Phi^{k}\left(\sum_{l=1}^{s}\alpha_{j}\Phi^{k_{j}}\circ\iota\left(q_{j}\right)\right)=\sum_{l=1}^{s}\alpha_{j}\Phi^{kk_{j}}\circ\iota\left(q_{j}\right)\in\mathcal{S}.
\]
\end{proof}
Next we explore individually the particular cases of $G$ equal to
$U\left(n\right)$, $SU\left(n\right)$ and $\mathrm{Sp}\left(n\right)$
to show that 
\[
\mathcal{S}=H^{*}\left(B_{\mathrm{com}}G,\mathbb{R}\right).
\]

\medskip{}

\subsection{Generators of $H^{*}\left(B_{com}U\left(n\right),\mathbb{R}\right)$: }

For this case recall that the Weil group of $U\left(n\right)$ is
isomorphic to the symmetric group $S_{n}$. By the previous section
we know that 
\[
H^{*}\left(B_{com}U\left(n\right),\mathbb{R}\right)=\left(\mathbb{R}\left[x_{1},\ldots,x_{n}\right]\otimes\mathbb{R}\left[y_{1},\ldots,y_{n}\right]\right)^{S_{n}}/J
\]
where $S_{n}$ acts diagonally on the tensor product, permuting the
variables of each factor. $J$ is the ideal generated by the symmetric
polynomials of positive degree on the $x_{i}$. It is also known that
\[
H^{*}\left(BU\left(n\right),\mathbb{R}\right)=\left(\mathbb{R}\left[x_{1},\ldots,x_{n}\right]\right)^{S_{n}},
\]
where the action is once again by permuting variables. $H^{*}\left(BU\left(n\right),\mathbb{R}\right)$
is generated by the power polynomials 
\[
p_{m}:=z_{1}^{m}+z_{2}^{m}+\cdots+z_{n}^{m},
\]
which are clearly invariant under the action of $S_{n}$. These polynomials
have their counterparts on two variables polynomials in the form of
\[
P_{a,b}\left(n\right):=x_{1}^{a}y_{1}^{b}+x_{2}^{a}y_{2}^{b}+\cdots+x_{n}^{a}y_{n}^{b},
\]
where $1\leq a+b\leq n$. These generate the algebra $\left(\mathbb{R}\left[x_{1},\ldots,x_{n}\right]\otimes\mathbb{R}\left[y_{1},\ldots,y_{n}\right]\right)^{S_{n}}$
(See \cite{key-11}, Theorem 1). Thus to prove that $\mathcal{S}$
is all of $H^{*}\left(B_{\mathrm{com}}U\left(n\right),\mathbb{R}\right)$
it is enough to see that the multisymmetric polynomials (modulo $J$)
are in fact in $\mathcal{S}$. To see it, we first need a couple of
lemmas.

\medskip{}

\begin{lem}
For every $n\in\mathbb{N}$ and $1\leq a+b\leq n$ with $a,b\geq0$
we have 
\[
\Phi^{k}\left(P_{a,b}\left(n\right)\right)=k^{b}P_{a,b}\left(n\right).
\]
\end{lem}

\begin{proof}
Since $\Phi^{k}$ is a homomorphism of algebras, we have 
\begin{align*}
\Phi^{k}\left(P_{a,b}\left(n\right)\right) & =\Phi^{k}\left(x_{1}^{a}y_{1}^{b}+x_{2}^{a}y_{2}^{b}+\cdots+x_{n}^{a}y_{n}^{b}\right)\\
 & =\sum_{i=1}^{n}\Phi^{k}\left(x_{i}^{a}y_{i}^{b}\right).
\end{align*}
But we have 
\begin{align*}
\Phi^{k}\left(x_{i}^{a}y_{i}^{b}\right) & =\Phi^{k}\left(x_{i}\right)^{a}\Phi^{k}\left(y_{i}\right)^{b}\\
 & =k^{b}x_{i}^{a}y_{i}^{b}.
\end{align*}
 Where the last equality is true since we already saw that $\Phi^{k}\left(x_{i}\right)=x_{i}$
and $\Phi^{k}\left(x_{i}\right)=ky_{i}$ for every $1\leq i\leq n$. 
\end{proof}
To prove the goal of this subsection, we illustrate explicitly the
cases $n=2$ and $n=3$.
\begin{itemize}
\item Suppose first that $n=2$.

We want to show that the following multisymmetric polynomials are
indeed in $\mathcal{S}$:
\begin{itemize}
\item $P_{0,1}\left(2\right)=y_{1}+y_{2}$,
\item $P_{1,1}\left(2\right)=x_{1}y_{1}+x_{2}y_{2}$ and
\item $P_{0,2}\left(2\right)=y_{1}^{2}+y_{2}^{2}$.
\end{itemize}
We ignore $P_{1,0}\left(2\right)=x_{1}+x_{2}$ since this is zero
modulo $J$. For this first observe that 
\[
\iota\left(z_{1}+z_{2}\right)=\left(x_{1}+y_{1}\right)+\left(x_{2}+y_{2}\right)=\left(x_{1}+x_{2}\right)+\left(y_{1}+y_{2}\right)=P_{1,0}\left(2\right)+P_{0,1}\left(2\right)
\]
clearly belongs to $\mathcal{S}$. Since $P_{1,0}\left(2\right)=0\,\mathrm{mod}\,J$
we are done. For $P_{1,1}\left(2\right)$ and $P_{0,2}\left(2\right)$
notice that the total degree (the sum of the power of each term) is
2, thus we have to consider $\iota\left(p_{2}\right)$: 
\begin{align*}
\iota\left(z_{1}^{2}+z_{2}^{2}\right) & =\left(x_{1}+y_{1}\right)^{2}+\left(x_{2}+y_{2}\right)^{2}\\
 & =\left(x_{1}^{2}+x_{2}^{2}\right)+2\left(x_{1}y_{1}+x_{2}y_{2}\right)+\left(y_{1}^{2}+y_{2}^{2}\right).
\end{align*}
This can be rewritten as 
\[
\iota\left(z_{1}^{2}+z_{2}^{2}\right)=P_{1,0}\left(2\right)+2P_{1,1}\left(2\right)+P_{0,2}\left(2\right).
\]
Then we consider
\begin{align*}
\Phi^{-1}\left(\iota\left(z_{1}^{2}+z_{2}^{2}\right)\right) & =\left(x_{1}^{2}+x_{2}^{2}\right)-2\left(x_{1}y_{1}+x_{2}y_{2}\right)+\left(y_{1}^{2}+y_{2}^{2}\right)
\end{align*}
giving us that 
\[
\iota\left(z_{1}^{2}+z_{2}^{2}\right)+\Phi^{-1}\left(\iota\left(z_{1}^{2}+z_{2}^{2}\right)\right)=2\left(y_{1}^{2}+y_{2}^{2}\right)\mathrm{mod}\,J
\]
meaning that $P_{0,2}\left(2\right)\in\mathcal{S}$, since $\mathcal{S}$
is a subalgebra closed under power maps. Finally modulo $J$ we get
\[
P_{1,1}\left(2\right)=\frac{\iota\left(z_{1}^{2}+z_{2}^{2}\right)-P_{0,2}\left(2\right)}{2}\in\mathcal{S}
\]
which finishes the proof for $n=2$.
\item Suppose now that $n=3$.

The arguments used in the case $n=2$ can be used to obtain the first
two of the next equalities, where once again they are taken to be
modulo $J$: 
\begin{enumerate}
\item $P_{0,1}\left(3\right)=\iota\left(z_{1}+z_{2}+z_{3}\right)$ .
\item $P_{0,2}\left(3\right)=\frac{1}{2}\left(\iota\left(z_{1}^{2}+z_{2}^{2}+z_{3}^{2}\right)+\Phi^{-1}\left(\iota\left(z_{1}^{2}+z_{2}^{2}+z_{3}^{2}\right)\right)\right)$.
\item $P_{1,1}\left(3\right)=\frac{1}{2}\left(\iota\left(z_{1}^{2}+z_{2}^{2}+z_{3}^{2}\right)-P_{0,2}\right)$.
\end{enumerate}
We are left to obtain $P_{a,b}\left(3\right)$ such that $a+b=3$.
For this we can reorder to see that
\begin{align*}
\iota\left(z_{1}^{3}+z_{2}^{3}+z_{3}^{3}\right) & =\left(x_{1}+y_{1}\right)^{3}+\left(x_{2}+y_{2}\right)^{3}+\left(x_{3}+y_{3}\right)^{3}\\
 & =\left(x_{1}^{3}+x_{2}^{3}+x_{3}^{3}\right)+3\left(x_{1}^{2}y_{1}+x_{2}^{2}y_{2}+x_{3}^{2}y_{3}\right)\\
 & \;+3\left(x_{1}y_{1}^{2}+x_{2}y_{2}^{2}+x_{3}y_{3}^{2}\right)+\left(y_{1}^{2}+y_{2}^{2}+y_{3}^{3}\right)
\end{align*}
which amounts to 
\[
\iota\left(z_{1}^{3}+z_{2}^{3}+z_{3}^{3}\right)=3P_{2,1}+3P_{1,2}+P_{0,3}\,\mathrm{mod}\,J.
\]
We use the power maps to get that 
\[
\Phi^{-1}\left(\iota\left(z_{1}^{3}+z_{2}^{3}+z_{3}^{3}\right)\right)=-3P_{2,1}+3P_{1,2}-P_{0,3}\,\mathrm{mod}J.
\]
By adding the last two equalities we get 
\[
P_{1,2}\mathrm{mod}J=\frac{1}{6}\left(\Phi^{-1}\left(\iota\left(z_{1}^{3}+z_{2}^{3}+z_{3}^{3}\right)\right)+\iota\left(z_{1}^{3}+z_{2}^{3}+z_{3}^{3}\right)\right)\in\mathcal{S}.
\]
 Thus we have $\iota\left(z_{1}^{3}+z_{2}^{3}+z_{3}^{3}\right)-3P_{1,2}\,\mathrm{mod}\,J\in\mathcal{S}$,
and by closure under power maps we obtain
\[
8P_{0,3}\mathrm{mod}J=\Phi^{2}\left(\iota\left(z_{1}^{3}+z_{2}^{3}+z_{3}^{3}\right)-3P_{1,2}\right)-6\left(\iota\left(z_{1}^{3}+z_{2}^{3}+z_{3}^{3}\right)-3P_{1,2}\right)\in\mathcal{S}
\]
from we conclude that $P_{0,3}\,\mathrm{mod}\,J\in\mathcal{S}$. We
finally have 
\[
P_{2,1}=\frac{1}{3}\left(\iota\left(z_{1}^{3}+z_{2}^{3}+z_{3}^{3}\right)-3P_{1,2}-P_{0,3}\right)\mathrm{mod}J
\]
which finishes the case $n=3$.
\end{itemize}
In the previous two examples we see that for non negative numbers
$a$ and $b$, we proved that $P_{a,b}\left(n\right)$ belongs to
$\mathcal{S}$ using induction on the value $a+b$. This was done
in such a way that the induction process did not depend on $n$. These
arguments can be generalized more methodically to obtain.

\medskip{}

\begin{thm}
\label{thm:Generators HBUn}The algebra $H^{*}\left(B_{com}U\left(n\right);\mathbb{R}\right)$
is equal to the subalgebra 
\[
\mathcal{S}:=\left\langle \Phi^{k}\left(\mathrm{Im}\iota\right)\mid k\in\mathbb{Z}\setminus\left\{ 0\right\} \right\rangle .
\]
\end{thm}

\begin{proof}
For this proof we will be working modulo $J$. Also, for an arbitrary
$n$ consider a fixed $m\in\left\{ 1,2,\ldots,n\right\} $. Now take
\[
p_{m}:=z_{1}^{m}+z_{2}^{m}+\cdots+z_{n}^{m}.
\]
An easy reordering gives us 
\begin{align}
\iota\left(p_{m}\right) & =\left(\sum_{i=1}^{n}\left(x_{i}+y_{i}\right)^{m}\right)=\sum_{i=1}^{n}\sum_{j=0}^{m}\binom{m}{j}x_{i}^{m-j}y_{i}^{j}\nonumber \\
 & =\sum_{j=0}^{m}\binom{m}{j}P_{m-j,j}\left(n\right)=\sum_{j=1}^{m}\binom{m}{j}P_{m-j,j}\left(n\right),\label{eq:principal sum}
\end{align}
where the last equality holds because we are working modulo $J$.
From this point we will use the power maps $\Phi^{k}$ to obtain the
various $P_{m-j},_{j}\left(n\right)$. First we use the following
recursion to get first $P_{0,m}\left(n\right)$ from \ref{eq:principal sum}:
Let $A_{0}:=\iota\left(p_{m}\right)$, 
\[
A_{1}:=\Phi^{2}\left(A_{0}\right)-2A_{0}=\sum_{j=2}^{m}\left(2^{j}-2\right)\binom{m}{j}P_{m-j,j}\left(n\right)
\]
and
\[
A_{2}:=\Phi^{3}\left(A_{1}\right)-3^{2}A_{1}=\sum_{j=3}^{m}\left(2^{j}-2\right)\left(3^{j}-3^{2}\right)\binom{m}{j}P_{m-j,j}\left(n\right).
\]
In general for $1\leq k\leq m-1$ we define 
\[
A_{k}:=\Phi^{k+1}\left(A_{k-1}\right)-\left(k+1\right)^{k}A_{k-1}.
\]
Notice that every $A_{k}$ has non zero coefficients only for $P_{m-j,j}\left(n\right)$
for $k+1\leq j\leq m$. Since $A_{0}\in\mathcal{S}$ by definition
and every $A_{k}$ is defined in terms of the power maps and $A_{k-1}$,
induction implies that $A_{k}\in\mathcal{S}$ for every $1\leq k\leq m-1$.
Some easy calculations allow us to obtain that 
\[
P_{0,m}\left(n\right)=\left(\prod_{k=2}^{m}\left(k^{m}-k^{k-1}\right)\right)^{-1}A_{m-1}\in\mathcal{S}.
\]

And thus we obtain that 
\[
\iota\left(p_{m}\right)-P_{0,m}\left(n\right)=\sum_{j=1}^{m-1}\binom{m}{j}P_{m-j,j}\left(n\right)\in\mathcal{S}.
\]
Then we can apply a new recursion to conclude that $P_{1,m-1}\left(n\right)\in\mathcal{S}$.
By continuing with this backwards recursion we obtain that $P_{a,b}\left(n\right)\in\mathcal{S}$
for all positive $a,b$ such that $a+b=m$. Since we picked $m\in\left\{ 1,2,\ldots,n\right\} $
arbitrarily, this finishes the proof. 
\end{proof}
\medskip{}

\subsection{Generators of $H^{*}\left(B_{com}SU\left(n\right);\mathbb{R}\right)$:}

To obtain that 
\[
H^{*}\left(B_{com}SU\left(n\right),\mathbb{R}\right)=\left\langle \Phi^{k}\left(\mathrm{Im}\iota\right)\mid k\in\mathbb{Z}\setminus\left\{ 0\right\} \right\rangle ,
\]
we use a different presentation of $H^{*}\left(BT,\mathbb{R}\right)$.
A maximal torus of $SU\left(n\right)$ is the set of diagonal matrices
with entries in $S^{1}\subseteq\mathbb{C}$, such that their product
equals one. Under such presentation it is routine to show that 
\[
H^{*}\left(BT,\mathbb{R}\right)\cong\left(\mathbb{R}\left[z_{1},\ldots,z_{n}\right]/\left\langle z_{1}+\cdots+z_{n}\right\rangle \right)
\]
where the Weyl group is then $S_{n}$ acting by permutation. This
implies that 
\[
H^{*}\left(BSU\left(n\right),\mathbb{R}\right)\cong\left(\mathbb{R}\left[z_{1},\ldots,z_{n}\right]/\left\langle z_{1}+\cdots+z_{n}\right\rangle \right)^{S_{n}},
\]
but since $p_{1}:=z_{1}+\cdots+z_{n}$ is already invariant, the previous
ring is isomorphic to 
\[
H^{*}\left(BSU\left(n\right),\mathbb{R}\right)\cong\mathbb{R}\left[z_{1},\ldots,z_{n}\right]^{S_{n}}/\left\langle z_{1}+\cdots+z_{n}\right\rangle .
\]
Since $\mathbb{R}\left[z_{1},\ldots,z_{n}\right]^{S_{n}}$ is itself
a polynomial algebra on $p_{i}=z_{1}^{i}+\cdots+z_{n}^{i}$ for $1\leq i\leq n$,
(see \cite{key-12}, Chapter 3.5: Chevalley's Theorem), we finally
get that 
\[
H^{*}\left(BSU\left(n\right),\mathbb{R}\right)\cong\mathbb{R}\left[p_{1},\ldots,p_{n}\right]/\left\langle p_{1}\right\rangle \cong\mathbb{R}\left[p_{2},\ldots,p_{n}\right].
\]
We will use this to conclude the following

\medskip{}

\begin{thm}
The real cohomology of $B_{\mathrm{com}}SU\left(n\right)$ can be
given by
\[
H^{\text{*}}\left(B_{com}SU\left(n\right),\mathbb{R}\right)\cong\left(\mathbb{R}\left[x_{1},\ldots,x_{n}\right]\otimes\mathbb{R}\left[y_{1},\ldots,y_{n}\right]\right)^{S_{n}}/\tilde{J},
\]
where $\tilde{J}$ is the ideal generated by $x_{1}^{i}+\cdots+x_{n}^{i}$,
$1\leq i\leq n$ and $y_{1}^{1}+\cdots+y_{n}^{1}$. 
\end{thm}

\begin{proof}
We saw in Theorem \ref{thm: Cohomology and Power maps} that 
\[
H^{\text{*}}\left(B_{com}SU\left(n\right),\mathbb{R}\right)\cong\left(H^{*}\left(BT\right)\otimes H^{*}\left(BT\right)\right)^{S_{n}}/J,
\]
where $J$ is the ideal generated by the $S_{n}$-invariants on the
first component. Now, for convenience, let us call 
\[
\mathbb{R}\left[\mathbf{x}\right]:=\mathbb{R}\left[x_{1},\ldots,x_{n}\right],
\]
\[
\mathbb{R}\left[\mathbf{y}\right]:=\mathbb{R}\left[y_{1},\ldots,y_{n}\right],
\]
\[
f\left(\mathbf{x}\right)=x_{1}+\cdots+x_{n}
\]
and
\[
f\left(\mathbf{y}\right)=y_{1}+\cdots+y_{n}
\]
 The previous reasoning then gives us
\[
H^{\text{*}}\left(B_{com}SU\left(n\right),\mathbb{R}\right)\cong\left(\mathbb{R}\left[\mathbf{x}\right]/\left\langle f\left(\mathbf{x}\right)\right\rangle \otimes\mathbb{R}\left[\mathbf{y}\right]/\left\langle f\left(\mathbf{y}\right)\right\rangle \right)^{S_{n}}/J.
\]
Notice that this is well defined since the $S_{n}$-invariance of
$x_{1}+\cdots+x_{2}$ and $y_{1}+\cdots+y_{2}$ allow us to have a
well define action of $S_{n}$ on 
\[
R:=\mathbb{R}\left[\mathbf{x}\right]/\left\langle f\left(\mathbf{x}\right)\right\rangle \otimes\mathbb{R}\left[\mathbf{y}\right]/\left\langle f\left(\mathbf{y}\right)\right\rangle .
\]
Consider first the map 
\[
p:\mathbb{R}\left[\mathbf{x}\right]\otimes\mathbb{R}\left[\mathbf{y}\right]\rightarrow R,
\]
which is induced by the projection

\[
\mathbb{R}\left[\mathbf{x}\right]\times\mathbb{R}\left[\mathbf{y}\right]\rightarrow\mathbb{R}\left[\mathbf{x}\right]/\left\langle f\left(\mathbf{x}\right)\right\rangle \times\mathbb{R}\left[\mathbf{y}\right]/\left\langle f\left(\mathbf{y}\right)\right\rangle .
\]
The map $p$ is naturally $S_{n}$-equivariant, thus it induces a
map 
\[
\tilde{p}:\left(\mathbb{R}\left[x_{1},\ldots,x_{n}\right]\otimes\mathbb{R}\left[y_{1},\ldots,y_{n}\right]\right)^{S_{n}}\rightarrow R^{S_{n}}.
\]
Also, since $p$ is surjective, and the action is diagonal, we have
that $\tilde{p}$ is also onto. We can further consider the composition
with the quotient by $J$ to obtain a surjective map 
\[
q:\left(\mathbb{R}\left[x_{1},\ldots,x_{n}\right]\otimes\mathbb{R}\left[y_{1},\ldots,y_{n}\right]\right)^{S_{n}}\rightarrow\left(R\right)^{S_{n}}/J.
\]
It is easy to see that the kernel of this map is what we called $\tilde{J}$,
so the result follows. 
\end{proof}
Even further, since the map 
\begin{align*}
\mathbb{R}\left[z_{1},\ldots,z_{n}\right] & \rightarrow\mathbb{R}\left[x_{1},\ldots,x_{n}\right]\otimes\mathbb{R}\left[y_{1},\ldots,y_{n}\right]\\
z_{i} & \mapsto x_{i}+y_{i}.
\end{align*}
induces the map $\iota:H^{\text{*}}\left(BSU\left(n\right),\mathbb{R}\right)\rightarrow H^{\text{*}}\left(B_{com}SU\left(n\right),\mathbb{R}\right)$,
we still have the same characterization under the identifications
given above. That is, $\iota$ can be seen as the map 
\[
\left(\mathbb{R}\left[z_{1},\ldots,z_{n}\right]/\left\langle z_{1}+\cdots+z_{n}\right\rangle \right)^{S_{n}}\rightarrow\left(\mathbb{R}\left[x_{1},\ldots,x_{n}\right]\otimes\mathbb{R}\left[y_{1},\ldots,y_{n}\right]\right)^{S_{n}}/\tilde{J}
\]
induce by $z_{i}\mapsto x_{i}+y_{i}$. The $k$-th power map on 
\[
\left(\mathbb{R}\left[x_{1},\ldots,x_{n}\right]\otimes\mathbb{R}\left[y_{1},\ldots,y_{n}\right]\right)^{S_{n}}/\tilde{J}
\]
 is also still induce by the assignment $x_{i}\mapsto x_{i}$ and
$y_{i}\mapsto ky_{i}$. Thus, with slight changes we can still apply
the arguments given in the proof of Theorem \ref{thm:Generators HBUn},
to obtain the main result. 

\medskip{}

\begin{thm}
The algebra $H^{*}\left(B_{com}SU\left(n\right);\mathbb{R}\right)$
is equal to the subalgebra 
\[
\mathcal{S}:=\left\langle \Phi^{k}\left(\mathrm{Im}\iota\right)\mid k\in\mathbb{Z}\setminus\left\{ 0\right\} \right\rangle ,
\]
where $\Phi^{k}$ is the $k$-th power map. 
\end{thm}

\medskip{}

\subsection{Generators of $H^{*}\left(B_{com}\mathrm{Sp}\left(n\right);\mathbb{R}\right)$:}

In this section $\mathbb{Z}_{2}$ will denote the multiplicative group
$\left\{ -1,1\right\} $. 

The Weyl group $W$ of the simplectic group $\mathrm{Sp}\left(n\right)$
is isomorphic to the semidirect product $\mathbb{Z}_{2}^{n}\rtimes S_{n}$,
where $\sigma\in S_{n}$ acts on $\left(a_{1},\ldots,a_{n}\right)\in\mathbb{Z}_{2}^{n}$
by 
\[
\sigma\cdot\left(a_{1},\ldots,a_{n}\right)=\left(a_{\sigma\left(1\right)},\ldots,a_{\sigma\left(n\right)}\right).
\]
Under these identifications, if 
\[
f\in\mathbb{R}\left[x_{1},\ldots,x_{n}\right]\cong H^{*}(T)
\]
 and 
\[
g=\left(\left(a_{1},\ldots,a_{n}\right),\sigma\right)\in\mathbb{Z}_{2}^{n}\rtimes S_{n}
\]
 we have 
\[
g\cdot f\left(x_{1},\ldots,x_{n}\right)=f\left(a_{1}x_{\sigma\left(1\right)},\ldots,a_{n}x_{\sigma\left(n\right)}\right).
\]
Recall that 
\[
H^{*}\left(B_{com}\mathrm{Sp}\left(n\right);\mathbb{R}\right)\cong\left(\mathbb{R}\left[x_{1},\ldots,x_{n}\right]\otimes\mathbb{R}\left[y_{1},\ldots,y_{n}\right]\right)^{W}/J
\]
where $W$ acts diagonally: for $n\in W$ and $p\left(x\right)\otimes q\left(y\right)\in\mathbb{R}\left[x_{1},\ldots,x_{n}\right]\otimes\mathbb{R}\left[y_{1},\ldots,y_{n}\right]$
we have
\[
n\cdot\left(p\left(x\right)\otimes q\left(y\right)\right):=\left(n\cdot p\left(x\right)\right)\otimes\left(n\cdot q\left(y\right)\right).
\]
$J$ is the ideal generated by the symmetric polynomials on the variables
$x_{i}^{2}$. For brevity, let us call $R:=\left(\mathbb{R}\left[x_{1},\ldots,x_{n}\right]\otimes\mathbb{R}\left[y_{1},\ldots,y_{n}\right]\right)^{W}$
the \textbf{signed multisymmetric polynomials}.

Once again we want to see that $\mathcal{S}:=\left\langle \Phi^{k}\left(\mathrm{Im}\iota\right)\mid k\in\mathbb{Z}\setminus\left\{ 0\right\} \right\rangle $
is equal to all of $H^{*}\left(B_{com}\mathrm{Sp}\left(n\right);\mathbb{R}\right)$.
For this let us see first that the set 
\[
\left\{ P_{a,b}\left(n\right)\mid a,b\geq0\:\mathrm{and}\:a+b\in2\mathbb{Z}\right\} 
\]
generates all of the signed multisymmetric polynomials as an algebra.
This will allow us to use the same arguments used in the case of $U\left(n\right)$
to obtain that $\mathcal{S}=H^{*}\left(B_{com}\mathrm{Sp}\left(n\right);\mathbb{R}\right)$.
We need the following lemmas, where the first has a straightforward
proof. 

\medskip{}

\begin{lem}
Let $\mu:\mathbb{R}\left[x_{1},\ldots,x_{n}\right]\otimes\mathbb{R}\left[y_{1},\ldots,y_{n}\right]\rightarrow R$
be the operator defined as 
\[
\mu\left(f\right)=\frac{1}{\left|W\right|}\sum_{g\in W}g\cdot f.
\]
This is a well defined $\mathbb{R}$-linear map, where $\left|W\right|$
is the cardinality of the Weyl group. We call this operator the \textbf{symmetrization
operator}. 
\end{lem}

\medskip{}

\begin{lem}
If $f\in R$ and $h\in\mathbb{R}\left[x_{1},\ldots,x_{n}\right]\otimes\mathbb{R}\left[y_{1},\ldots,y_{n}\right]$,
then $\mu\left(f\right)=f$ and $\mu\left(fh\right)=f\cdot\mu\left(h\right)$.
\end{lem}

\begin{proof}
Since $f$ is invariant, we have that $g\cdot f=f$ for all $g\in W$,
thus 
\[
\mu\left(f\right)=\frac{1}{\left|W\right|}\sum_{g\in W}g\cdot f=\frac{\left|W\right|}{\left|W\right|}f=f.
\]
Also, since by definition $g\cdot\left(fh\right)=\left(g\cdot f\right)\left(g\cdot h\right)$
for every $g\in W$ and $f,h\in\mathbb{R}\left[x_{1},\ldots,x_{n}\right]\otimes\mathbb{R}\left[y_{1},\ldots,y_{n}\right]$.
In particular if $f$ is invariant it follows that 
\[
\mu\left(fh\right)=\frac{1}{\left|W\right|}\sum_{g\in W}g\cdot\left(fh\right)=\frac{f}{\left|W\right|}\sum_{g\in W}g\cdot h=f\cdot\mu\left(h\right).
\]
\end{proof}
In order to prove our objective we need to analyze the summands (or
monomials) of a signed multisymmetric polynomials first. For this
consider sets of indices $I=\left(i_{1},\ldots,i_{n}\right),J=\left(j_{1},\ldots,j_{n}\right)\in\mathbb{N}^{n}$
(including zero as a natural number) and let us denote
\[
x^{I}y^{J}:=x_{1}^{i_{1}}x_{2}^{i_{2}}\cdots x_{n}^{i_{n}}y_{1}^{j_{1}}\cdots y_{n}^{j_{n}}.
\]

\medskip{}

\begin{defn}
We say a pair of multi indices $\left(I,J\right)\in\mathbb{N}^{n}\times\mathbb{N}^{n}$
is \textbf{odd} if there if $1\leq k\leq n$ such that $i_{k}+j_{k}$
is odd. Such a pair is \textbf{even} if it is not odd. 
\end{defn}

\medskip{}

\begin{lem}
If a pair of multi indices $\left(I,J\right)$ is odd, then $\mu\left(x^{I}y^{J}\right)=0$. 
\end{lem}

\begin{proof}
Let $\left(I,J\right)=\left(\left(i_{1},\ldots,i_{n}\right),\left(j_{1},\ldots,j_{n}\right)\right)$
and let's assume $i_{k}+j_{k}$ is odd. Let 
\[
h_{k}:=\left(\left(1,\ldots,1,\underbrace{-1}_{k-\mathrm{position}},1,\ldots,1\right),e\right)\in W,
\]
where $e$ is the identity permutation. Denote by $H\subseteq W$
the subgroup generated by $h_{k}$ and the partition by right cosets
$\left\{ Hg_{1},\ldots,Hg_{m}\right\} $ of $W.$ Since $h_{k}$ has
order 2 
\[
W=\left\{ g_{1},\ldots,g_{m}\right\} \cup\left\{ h_{k}g_{1},\ldots,h_{k}g_{m}\right\} 
\]
and thus
\[
\mu\left(x^{I}y^{J}\right)=\frac{1}{\left|W\right|}\sum_{l=1}^{m}\left(g_{l}x^{I}y^{J}+h_{k}\left(g_{l}x^{I}y^{J}\right)\right).
\]
Notice that in general if $g=\left(\left(a_{1},\ldots,a_{n}\right),\sigma\right)$,
then since $i_{k}+j_{k}$ is odd we get

\begin{align*}
h_{k}\left(g\cdot x^{I}y^{J}\right) & =h_{k}\left(a_{1}^{i_{1}+j_{1}}\cdots a_{n}^{i_{n}+j_{n}}x_{\sigma\left(1\right)}^{i_{1}}\cdots x_{\sigma\left(n\right)}^{i_{n}}y_{\sigma\left(1\right)}^{j_{1}}\cdots y_{\sigma\left(n\right)}^{j_{n}}\right)\\
 & =\left(-1\right)^{i_{k}+j_{k}}a_{1}^{i_{1}+j_{1}}\cdots a_{n}^{i_{n}+j_{n}}x_{\sigma\left(1\right)}^{i_{1}}\cdots x_{\sigma\left(n\right)}^{i_{n}}y_{\sigma\left(1\right)}^{j_{1}}\cdots y_{\sigma\left(n\right)}^{j_{n}}\\
 & =-a_{1}^{i_{1}+j_{1}}\cdots a_{n}^{i_{n}+j_{n}}x_{\sigma\left(1\right)}^{i_{1}}\cdots x_{\sigma\left(n\right)}^{i_{n}}y_{\sigma\left(1\right)}^{j_{1}}\cdots y_{\sigma\left(n\right)}^{j_{n}}\\
 & =-g\cdot x^{I}y^{J}.
\end{align*}
This implies that 
\[
\mu\left(x^{I}y^{J}\right)=\frac{1}{\left|W\right|}\sum_{l=1}^{m}\left(g_{l}x^{I}y^{J}-g_{l}x^{I}y^{J}\right)=0.
\]
\end{proof}
\begin{thm}
If a polynomial is signed multisymmetric then its monomials have all
even multi indices. 
\end{thm}

\begin{proof}
An element $f\in\mathbb{R}\left[x_{1},\ldots,x_{n}\right]\otimes\mathbb{R}\left[y_{1},\ldots,y_{n}\right]$
can be uniquely written as
\[
f=c_{0}+\sum_{k=1}^{m}c_{k}x^{I_{k}}y^{J_{k}}.
\]
Where $c_{0}\in\mathbb{R}$ and for $k>0$, $c_{k}\in\mathbb{R}\setminus\left\{ 0\right\} $,
$I_{k}$ and $J_{k}$ are multi indices of $n$ variables, not all
of them zero. If $f$ is signed multisymmetric, 
\[
f=\mu\left(f\right)=c_{0}+\sum_{k=1}^{m}c_{k}\mu\left(x^{I_{k}}y^{J_{k}}\right).
\]
These two last expressions for $f$ imply that
\begin{equation}
\sum_{k=1}^{m}c_{k}x^{I_{k}}y^{J_{k}}=\sum_{k=1}^{m}c_{k}\mu\left(x^{I_{k}}y^{J_{k}}\right).\label{eq:two expressions}
\end{equation}
 But by the previous lemma, we know that if $\left(I_{t},J_{t}\right)$
is odd for a given $t$, then $\mu\left(x^{I_{t}}y^{J_{t}}\right)=0$.
Since $\mu\left(x^{I_{k}}y^{J_{k}}\right)$ is itself a sum of monomials,
the expression 
\[
\sum_{k=1}^{m}c_{k}\mu\left(x^{I_{k}}y^{J_{k}}\right)
\]
 must have only monomials with an even set of multi indices. Since
all the coefficients in 
\[
\sum_{k=1}^{m}c_{k}x^{I_{k}}y^{J_{k}}
\]
are non zero, the last equality and the uniqueness of the expression
for non zero coefficients of a polynomial, allow us to conclude that
$\left(I_{k},J_{k}\right)$ is even for every $1\leq k\leq m$.
\end{proof}
In particular this proof allows us to obtain 

\medskip{}

\begin{cor}
\label{cor: Signed symmetric expresion}Every signed multisymmetric
polynomial can be written in the form 
\[
f=c_{0}+\sum_{k=1}^{m}c_{k}\mu\left(x^{I_{k}}y^{J_{k}}\right),
\]
 where $\left(I_{k},J_{k}\right)$ is even for every $1\leq k\leq m$.
\end{cor}

If a multisymmetric polynomial has monomials with even multi indices,
such polynomial is signed symmetric, meaning that is invariant under
the action of elements of the form $\left(\left(a_{1},\ldots,a_{n}\right),e\right)\in W$.
In particular we can now conclude:

\medskip{}

\begin{thm}
A multisymmetric polynomial is signed symmetric if only if all its
monomials have even multi indices.
\end{thm}

This result grant us the frame work to obtain generators for the algebra
\[
H^{*}\left(B_{com}\mathrm{Sp}\left(n\right);\mathbb{R}\right).
\]
 Recall that multisymmetric are generated by the power polynomials
\[
P_{a,b}:=\sum_{i=1}^{n}x_{i}^{a}y_{i}^{b}.
\]
On the other hand, due to the last result we know $P_{a,b}$ is signed
multi symmetric if and only if $a+b$ is even. Let's see that they
in fact generate all of the signed multisymmetric polynomials. 
\begin{thm}
$\left(\mathbb{R}\left[x_{1},\ldots,x_{n}\right]\otimes\mathbb{R}\left[y_{1},\ldots,y_{n}\right]\right)^{\mathbb{Z}_{2}^{n}\rtimes S_{n}}$
is generated as an algebra by the set
\[
\mathcal{G}:=\left\{ P_{a,b}:=\sum_{i=1}^{n}x_{i}^{a}y_{i}^{b}\mid0\leq a,b\:\mathrm{and}\:a+b\in2\mathbb{Z}\right\} .
\]
\end{thm}

\begin{proof}
By Corollary \ref{cor: Signed symmetric expresion} is enough to show
that for even multi indices $\left(I,J\right)$, 
\[
\mu\left(x^{I}y^{J}\right)\in\mathrm{gen}\mathcal{G}.
\]
 To see this, note that any permutation of the set of indices have
the same symmetrization. This is, for $k_{1},\ldots,k_{p}\in\left\{ 1,\ldots,n\right\} $
all mutually different, $p\leq n$, we have
\[
\mu\left(x_{k_{1}}^{i_{1}}\cdots x_{k_{p}}^{i_{p}}y_{k_{1}}^{j_{1}}\cdots y_{k_{p}}^{j_{p}}\right)=\mu\left(x_{1}^{i_{1}}\cdots x_{p}^{i_{p}}y_{1}^{j_{1}}\cdots y_{p}^{j_{p}}\right).
\]
So it is enough to show that 
\[
\mu\left(x_{1}^{i_{1}}\cdots x_{p}^{i_{p}}y_{1}^{j_{1}}\cdots y_{p}^{j_{p}}\right)\in\mathrm{gen}\mathcal{G},
\]
where of course $i_{k}+j_{k}$ is even for every $1\leq k\leq p$.
We do it using induction on $p$. The cases $p=1$ is immediate, since
in this case $\mu\left(x^{I}y^{J}\right)$ is a scalar multiple of
even power polynomials of the form $P_{a,0},P_{0,b}$ or $P_{a,b}$.

Next, assume we know $\mu\left(x_{1}^{i_{1}}\cdots x_{p}^{i_{p}}y_{1}^{j_{1}}\cdots y_{p}^{j_{p}}\right)\in\mathrm{gen}\mathcal{G}$
for $1\leq p\leq k$. By reordering we have 
\[
\mu\left(x_{1}^{i_{1}}y_{1}^{j_{1}}\right)\mu\left(x_{2}^{i_{2}}\cdots x_{k+1}^{i_{k+1}}y_{2}^{j_{2}}\cdots y_{k+1}^{j_{k+1}}\right)=c\mu\left(x_{1}^{i_{1}}\cdots x_{k+1}^{i_{k+1}}y_{1}^{j_{1}}\cdots y_{k+1}^{j_{k+1}}\right)+\Theta,
\]
with
\[
\Theta=\sum_{r=2}^{k+1}c_{r}\mu\left(x_{2}^{i_{2}}\cdots x_{r}^{i_{r}+i_{1}}\cdots x_{k+1}^{i_{k+1}}y_{2}^{j_{2}}\cdots y_{r}^{j_{r}+j_{1}}\cdots y_{k+1}^{j_{k+1}}\right),
\]
where $c_{2},\ldots,c_{k+1}$ are integers, and $c$ is a non zero
integer. This implies that 
\begin{align*}
\mu\left(x_{1}^{i_{1}}\cdots x_{k+1}^{i_{k+1}}y_{1}^{j_{1}}\cdots y_{k+1}^{j_{k+1}}\right) & =\frac{1}{c}\left(\mu\left(x_{1}^{i_{1}}y_{1}^{j_{1}}\right)\mu\left(x_{2}^{i_{2}}\cdots x_{k}^{i_{k}}y_{2}^{j_{2}}\cdots y_{k}^{j_{k}}\right)-\right.\\
 & \left.\sum_{r=2}^{k+1}c_{r}\mu\left(x_{2}^{i_{2}}\cdots x_{r}^{i_{r}+i_{1}}\cdots x_{k}^{i_{k}}y_{2}^{j_{2}}\cdots y_{r}^{j_{r}+j_{1}}\cdots y_{k}^{j_{k}}\right)\right).
\end{align*}
By the induction hypothesis all of the terms in the right are in $\mathrm{gen}\mathcal{G}$,
which implies that 
\[
\mu\left(x_{1}^{i_{1}}\cdots x_{k+1}^{i_{k+1}}y_{1}^{j_{1}}\cdots y_{k+1}^{j_{k+1}}\right)
\]
 belongs to $\mathrm{gen}\mathcal{G}$.
\end{proof}
With the last result at hand we can imitate the reasoning in the proof
of Theorem \ref{thm:Generators HBUn} to obtain the main result of
this part. 

\medskip{}

\begin{thm}
\label{thm:generatos H of Sp}The algebra $H^{*}\left(B_{com}\mathrm{Sp}\left(n\right);\mathbb{R}\right)$
is equal to the subalgebra 
\[
\mathcal{S}:=\left\langle \Phi^{k}\left(\mathrm{Im}\iota\right)\mid k\in\mathbb{Z}\setminus\left\{ 0\right\} \right\rangle .
\]
Where $\Phi^{k}$ are the power maps and $\iota:H^{*}\left(B\mathrm{Sp}\left(n\right);\mathbb{R}\right)\rightarrow H^{*}\left(B_{com}\mathrm{Sp}\left(n\right);\mathbb{R}\right)$
is the map induced by the homomorphism
\begin{align*}
\mathbb{R}\left[z_{1},\ldots,z_{n}\right] & \rightarrow\mathbb{R}\left[x_{1},\ldots,x_{n}\right]\otimes\mathbb{R}\left[y_{1},\ldots,y_{n}\right]\\
z_{i} & \mapsto x_{i}+y_{i}.
\end{align*}
\end{thm}

\begin{proof}
Take once again
\[
p^{m}=z_{1}^{m}+z_{2}^{m}+\cdots+z_{n}^{m}\in\mathbb{R}\left[z_{1},\ldots,z_{n}\right]
\]
for $m$ even. We also work modulo $J$, the ideal generated by the
$x_{i}^{2}$. Recall that
\[
\iota\left(p^{m}\right)=\sum_{j=1}^{m}\binom{m}{j}P_{m-j,j}\left(n\right).
\]
Since $\left(m-j\right)+j=m$, all of the power polynomials $P_{m-j,j}\left(n\right)$
are even. Now we use recursion to get first $P_{0,m}\left(n\right)$
from the last equality: for this we name $A_{0}:=\iota\left(p_{m}\right)$,
then we take
\[
A_{1}:=\Phi^{2}\left(A_{0}\right)-2A_{0}=\sum_{j=2}^{m}\left(2^{j}-2\right)\binom{m}{j}P_{m-j,j}\left(n\right)
\]
and
\[
A_{2}:=\Phi^{3}\left(A_{1}\right)-3^{2}A_{1}=\sum_{j=3}^{m}\left(2^{j}-2\right)\left(3^{j}-3^{2}\right)\binom{m}{j}P_{m-j,j}\left(n\right).
\]
In general for $1\leq k\leq m-1$ we define 
\[
A_{k}:=\Phi^{k+1}\left(A_{k-1}\right)-\left(k+1\right)^{k}A_{k-1}.
\]
Notice that every $A_{k}$ has non zero coefficients only for $P_{m-j,j}\left(n\right)$
for $k+1\leq j\leq m$. Since $A_{0}\in\mathcal{S}$ by definition
and every $A_{k}$ is defined in terms of the power maps and $A_{k-1}$,
induction implies that $A_{k}\in\mathcal{S}$ for every $1\leq k\leq m-1$.
In particular we have 
\[
P_{0,m}\left(n\right)=\left(\prod_{k=2}^{m}\left(k^{m}-k^{k-1}\right)\right)^{-1}A_{m-1}\in\mathcal{S}.
\]

We now can apply a similar procedure to 
\[
\iota\left(p_{m}\right)-P_{0,m}\left(n\right)=\sum_{j=1}^{m-1}\binom{m}{j}P_{m-j,j}\left(n\right)\in\mathcal{S}
\]
to conclude that if $m=2k$, and $P_{a,b}$ is such that $a+b=m$
then $P_{a,b}\in\mathcal{S}.$
\end{proof}
\medskip{}

\section{Chern-Weil theory for TC structures }

In this section we develop characteristic classes for TC structures.
Our central goal is to obtain characteristic classes for TC structures
using Chern-Weil theory. Specifically, we will develop this for TC
structures over vector bundles whose structural group is either $U\left(n\right)$
or $SU\left(n\right)$. Thus, by $G$ we will mean one of these groups. 

Recall that we have the $k$-th power maps $\Phi^{k}:H^{*}\left(B_{\mathrm{com}}G\right)\rightarrow H^{*}\left(B_{\mathrm{com}}G\right)$,
and a natural inclusion $\iota:H^{*}\left(BG\right)\rightarrow H^{*}\left(B_{\mathrm{com}}G\right)$.
We already proved that 

\medskip{}

\begin{thm}
For $G$ equal to $U\left(n\right)$ or $SU\left(n\right)$, then
\[
H^{*}\left(B_{\mathrm{com}}G\right)=\mathrm{gen}\left\{ \Phi^{k}\circ\iota\left(c\right)\mid c\in H^{*}(BG),k\in\mathbb{Z}\setminus\left\{ 0\right\} \right\} .
\]
\end{thm}

This means that given a class in $H^{*}\left(B_{\mathrm{com}}G\right)$,
it can be written as a sum of finite products of elements of the form
$\Phi^{k}\circ\iota\left(c\right)$, $c\in H^{*}(BG),k\in\mathbb{Z}\setminus\left\{ 0\right\} $.
Now we continue with a construction that allow us to use Chern-Weil
theory to compute characteristic classes associated to the previous
classes. 

\medskip{}

\subsection{Chern-Weil theory for TC structures:}

For the rest of this section let $\varepsilon\in\left[M,B_{\mathrm{com}}G\right]$
be an equivalence class with an underlying smooth vector bundle $E\rightarrow M$
and structure group $U\left(n\right)$ or $SU\left(n\right)$. For
an element $p\in H^{*}\left(B_{\mathrm{com}}G,\mathbb{R}\right)$
we denote by $p\left(\varepsilon\right)\in H^{*}\left(M\right)$ the
value of the TC characteristic class on the TC equivalence class $\varepsilon$.
Also, recall that via the Chern-Weil isomorphism, if $\mathfrak{g}$
is the Lie algebra of $G$, then $H^{*}\left(BG\right)\cong I\left(\mathfrak{g}\right)$.
Here $I\left(\mathfrak{g}\right)$ is the subalgebra of invariant
polynomials under conjugation of the polynomial algebra of $\mathfrak{g}$.
Under this identification, every characteristic class for vector bundles
(having $G$ as its structure group) can be identified with a polynomial
$c\in I\left(\mathfrak{g}\right)$. 

Now recall that for a smooth vector bundle $F\rightarrow M$ with
curvature $\Omega$, the value on $F$ of the characteristic class
associated to $c$ is equal to $c\left(\Omega\right)\in H^{*}\left(M\right)$.
Under these terms, we are now able to compute the TC characteristic
classes associated to the set of generators of $H^{*}\left(B_{\mathrm{com}}G\right)$,
$\left\{ \Phi^{k}\circ\iota\left(c\right)\mid1\leq i\leq n,k\in\mathbb{Z}\setminus\left\{ 0\right\} \right\} $.
Here, we take $\iota$ to be a map from $I\left(\mathfrak{g}\right)$
to $H^{*}\left(B_{\mathrm{com}}G\right)$.

\medskip{}

\begin{thm}
Consider $\varepsilon\in\left[M,B_{\mathrm{com}}G\right]$ an equivalence
class with an underlying smooth vector bundle $E\rightarrow M$, and
structure group $U\left(n\right)$ or $SU\left(n\right)$. Also let
$\Omega_{k}$ be the curvature of $E^{k}$, the $k$-th associated
bundle of $E$. Then for $c\in I\left(\mathfrak{g}\right)$ and $p=\Phi^{k}\circ\iota\left(c\right)\in H^{*}\left(B_{\mathrm{com}}G\right)$,
the TC characteristic class $p\left(\varepsilon\right)$ has same
class in $H^{*}\left(M\right)$ as the characteristic class for vector
bundles $c\left(E^{k}\right)$. This implies that
\[
p\left(\varepsilon\right)=c\left(\Omega_{k}\right)\in H^{*}\left(M\right).
\]
\end{thm}

\begin{proof}
This is straight forward. First, by Theorem \ref{thm:(Classifying-functions-for assiciated bundles}
we know that if $f$ and $f_{k}$ the the classifying functions for
TC structures over $E\rightarrow M$ and $E^{k}\rightarrow M$, respectively,
then there is the following commuting diagram
\[
\xymatrix{H^{*}\left(B_{\mathrm{com}}G\right)\ar[r]^{f^{*}} & H^{*}\left(M\right)\\
H^{*}\left(B_{\mathrm{com}}G\right)\ar[ru]^{f_{k}^{*}}\ar[u]^{\Phi^{k}}.
}
\]
This means that for $c\in H^{*}\left(BG\right)$ we have the identity
$f^{*}\left(\Phi^{k}\circ\iota\left(c\right)\right)=f_{k}^{*}\left(\iota\left(c\right)\right)$
in $H^{*}\left(M\right)$. 

In turn, since the composition $f_{k}^{*}\circ\iota$ is a classifying
function for the vector bundle $E^{k}\rightarrow M$, we can apply
the Chern-Weil isomorphism. That is, we can consider the curvature
$\Omega_{k}$ of $E_{k}$ to obtain that 
\[
f_{k}^{*}\left(\iota\left(c\right)\right)=c\left(\Omega_{k}\right).
\]
The conclusion of the theorem then follows by transitivity. 
\end{proof}
\medskip{}

\begin{thm}
(Chern-Weil theory for TC structures) Consider $\varepsilon\in\left[M,B_{\mathrm{com}}G\right]$
an equivalence class with an underlying smooth vector bundle $E\rightarrow M$,
and structure group $U\left(n\right)$ or $SU\left(n\right)$. Also
let $\Omega_{k}$ be the curvature of $E^{k}$, the $k$-th associated
bundle of $E$. Then every TC characteristic class can be obtained
as a linear combinations of products of the form 
\[
s_{1}\left(\Omega_{k_{1}}\right)\cdot s_{2}\left(\Omega_{k_{2}}\right)\cdots s_{m}\left(\Omega_{k_{m}}\right)\in H^{*}\left(M\right),
\]
where $s_{i}\in H^{*}\left(BG\right)$ and $k_{i}\in\mathbb{Z}$.
Each $s_{i}\left(\Omega_{k_{1}}\right)$ is the characteristic class
of the vector bundle $E^{k}\rightarrow M$ computed using its curvature. 
\end{thm}

\begin{proof}
Recall that if we set $\mathcal{S}$ as the subalgebra of $H^{*}\left(B_{\mathrm{com}}G\right)$
generated by 
\[
\left\{ \Phi^{k}\circ\iota\left(s\right)\mid1\leq i\leq n,k\in\mathbb{Z}\setminus\left\{ 0\right\} ,s\in H^{*}\left(BG\right)\right\} 
\]
then we have $\mathcal{S}=H^{*}\left(B_{\mathrm{com}}G\right)$. Thus,
every element of $H^{*}\left(B_{\mathrm{com}}G\right)$ can be written
as a linear combination of products of the form
\[
\Phi^{k_{1}}\left(\iota\left(s_{1}\right)\right)\cdot\Phi^{k_{m}}\left(\iota\left(s_{2}\right)\right)\cdots\Phi^{k_{m}}\left(\iota\left(s_{m}\right)\right).
\]
Then we can apply the previous theorem to obtain 
\[
\Phi^{k_{i}}\left(\iota\left(s_{i}\right)\right)=s_{i}\left(\Omega_{k_{i}}\right).
\]
\end{proof}
As suggested by the name of the theorem, we are now able to compute
TC characteristic classes by using Chern-Weil theory. This is done
in a three steps process for a class in $s\in H^{*}\left(B_{\mathrm{com}}G\right)$
and a TC structure $\xi$ over a vector bundle $E\rightarrow M$:
first we need to decompose $s$ in terms of the generators in 
\[
\left\{ \Phi^{k}\circ\iota\left(c\right)\mid1\leq i\leq n,k\in\mathbb{Z}\setminus\left\{ 0\right\} \right\} .
\]
 Secondly, for each of the generators $\Phi^{k}\circ\iota\left(c\right)$
in the decomposition of $s$ we use the curvature of the $k$-th associated
bundle, $\Omega_{k}$, to compute the characteristic class associated
to it, $c\left(\Omega_{k}\right)\in H^{*}\left(M\right)$ (this class
is equal to the TC class given by $\left(\Phi^{k}\circ\iota\left(c\right)\right)\left(\xi\right)$).
Finally we replace the values of each $\left(\Phi^{k}\circ\iota\left(c\right)\right)\left(\xi\right)$
to obtain $s\left(\xi\right)\in H^{*}\left(M\right)$. 

Recall from Chapter 3 that when $G$ is equal to $U\left(n\right)$,
then 
\[
H^{*}\left(B_{\mathrm{com}}G,\mathbb{R}\right)\cong\left(\mathbb{R}\left[x_{1},\ldots,x_{n}\right]\otimes\mathbb{R}\left[y_{1},\ldots,y_{n}\right]\right)^{S_{n}}/J
\]
where $S_{n}$ acts by permutation on their indexes and $J$ is the
ideal generated by the invariant polynomials of positive degree on
the $x_{i}$. When $G$ is $SU\left(n\right)$ is the same description
for $H^{*}\left(B_{\mathrm{com}}G,\mathbb{R}\right)$ except $J$
is generated by the invariant polynomials of positive degree on $x_{i}$
and the polynomial $y_{1}+\cdots y_{n}$. 

We also have the identifications 
\[
H^{*}\left(BU\left(n\right),\mathbb{R}\right)\cong\mathbb{R}\left[z_{1},\ldots,z_{n}\right]^{S_{n}}
\]
 and 
\[
H^{*}\left(BSU\left(n\right),\mathbb{R}\right)\cong\mathbb{R}\left[z_{1},\ldots,z_{n}\right]^{S_{n}}/\left\langle z_{1}+\dots+z_{n}\right\rangle .
\]
 Then we have that the polynomials 
\[
p_{i}=z_{1}^{i}+\cdots+z_{n}^{i}\in\mathbb{R}\left[z_{1},\ldots,z_{n}\right]
\]
generated all of $H^{*}\left(BG,\mathbb{R}\right)$, when $G$ is
$U\left(n\right)$ or $SU\left(n\right)$. Even further for $a,b\in\mathbb{N}\cup\left\{ 0\right\} $
such that $0<a+b$ then 
\[
P_{a,b}\left(n\right):=\sum_{i=1}^{n}x^{a}y^{b}\:\mathrm{mod}\:J.
\]
 generated all of $H^{*}\left(B_{\mathrm{com}}G,\mathbb{R}\right)$
as an algebra. We also saw in the proof of Theorem \ref{thm:Generators HBUn}
there every $P_{a,b}\left(n\right)$ can be obtain, via a recursive
procedure, as a linear combination of elements of the form $\Phi^{k}\left(\iota\left(p_{i}\right)\right)$.
With that recursive procedure and the previous theorem, we can compute
the TC characteristic classes corresponding to each $P_{a,b}\left(n\right)$.

Recall that another set of generators for $H^{*}\left(BG,\mathbb{R}\right)$,
when $G$ is $U\left(n\right)$ or $SU\left(n\right)$ is given by
the polynomials $\sigma_{i}$, characterized by the equation 
\[
\det\left(I-tX\right)=1+t\sigma_{1}\left(X\right)+t^{2}\sigma_{2}\left(X\right)+\cdots+t^{n}\sigma_{n}\left(X\right).
\]
These generators are more commonly used instead of the $p_{i}$, as
$\sigma_{i}$ are used in the definition of Chern classes. 

\medskip{}

\begin{example}
We saw previously that for $G=U\left(3\right)$ we have the equalities 
\begin{enumerate}
\item $y_{1}+y_{2}+y_{3}=\iota\left(z_{1}+z_{2}+z_{3}\right)$.
\item $y_{1}^{2}+y_{2}^{2}+y_{3}^{2}=\frac{1}{2}\left(\iota\left(z_{1}^{2}+z_{2}^{2}+z_{3}^{2}\right)+\Phi^{-1}\left(\iota\left(z_{1}^{2}+z_{2}^{2}+z_{3}^{2}\right)\right)\right)$.
\item $x_{1}y_{1}+x_{2}y_{2}+x_{3}y_{3}=\frac{1}{4}\left(\iota\left(z_{1}^{2}+z_{2}^{2}+z_{3}^{2}\right)-\Phi^{-1}\left(\iota\left(z_{1}^{2}+z_{2}^{2}+z_{3}^{2}\right)\right)\right)$.
\end{enumerate}
Now consider a TC strcutrure $\xi$ with underlying vector bundle
$E\rightarrow M$, with curvature $\Omega$ and $\Omega_{k}$ is the
curvature of the $k$-th associated bundle. Now since we have that
$p_{1}=\sigma_{1}$ and that
\[
p_{2}=\sigma_{1}^{2}-2\sigma_{2}
\]
we obtain that 
\begin{enumerate}
\item $\left(y_{1}+y_{2}+y_{3}\right)\left(\xi\right)=\sigma_{1}\left(\Omega\right)$.
\item $\left(y_{1}^{2}+y_{2}^{2}+y_{3}^{2}\right)\left(\xi\right)=\frac{1}{2}\left(\sigma_{1}\left(\Omega\right)^{2}+\sigma_{1}\left(\Omega_{-1}\right)^{2}\right)-\left(\sigma_{2}\left(\Omega\right)+\sigma_{2}\left(\Omega_{-1}\right)\right)$.
\item And finally 
\begin{align*}
\left(x_{1}y_{1}+x_{2}y_{2}+x_{3}y_{3}\right)\left(\xi\right)= & \frac{1}{4}\left(\sigma_{1}\left(\Omega\right)^{2}-\sigma_{1}\left(\Omega_{-1}\right)^{2}\right)\\
 & +\frac{1}{2}\left(\sigma_{2}\left(\Omega_{-1}\right)-\sigma_{2}\left(\Omega\right)\right).
\end{align*}
\end{enumerate}
\end{example}

For $G=U\left(n\right)$ we know that $H^{*}\left(BG\right)$ is a
polynomial algebra generated by the Chern classes $c_{i}$, $1\leq i\leq n$.
Thus it follows that $\mathcal{S}$ is generated by the set $\left\{ \Phi^{k}\circ\iota\left(c_{i}\right)\mid1\leq i\leq n,k\in\mathbb{Z}\setminus\left\{ 0\right\} \right\} $. 

\medskip{}

\begin{defn}
We call the classes of the form $c_{i}^{k}:=\Phi^{k}\circ\iota\left(c\right)\in H^{*}\left(B_{\mathrm{com}}U\left(n\right)\right)$
the \textbf{TC Chern classes}. Also, for a TC structure $\varepsilon$
with underlying vector bundle $E\rightarrow M$ we call 
\[
c_{i}^{k}\left(\varepsilon\right):=f^{*}\left(c_{i}^{k}\right)\in H^{*}\left(M\right)
\]
the TC $\left(i,k\right)$-Chern class. Here $f:M\rightarrow B_{\mathrm{com}}U\left(n\right)$
is the classifying function of the TC structure. 
\end{defn}

From the previous theorem we have the immediate following corollary:

\medskip{}

\begin{cor}
Let $E\rightarrow M$ by the underlying bundle of a TC structure structure
$\varepsilon$, and let $\Omega_{k}$ be the curvature of the $k$-th
associated bundle. Then $c_{i}^{k}(\epsilon)=c_{i}\left(\Omega_{k}\right)$.
\end{cor}

It is immediate from our results that 
\[
\left\{ c_{i}^{k}\mid k\in\mathbb{Z},\:i\in\mathbb{N}\right\} 
\]
 generates all of $B_{\mathrm{com}}U\left(n\right)$ as an algebra.
That is, every class in $H^{*}\left(B_{\mathrm{com}}U\left(n\right)\right)$
can be written in the form 
\[
s=\sum_{j=1}^{m}\alpha_{j}C_{j}
\]
where $\alpha_{j}\in\mathbb{R}$ and 
\[
C_{j}=\prod_{t=1}^{m_{j}}c_{i_{j,t}}^{k_{j,t}},
\]
where $k_{j,t}\in\mathbb{Z}$ and $i_{j,t}\in\mathbb{N}$. Then it
follows that if $\xi$ is a TC structure with underlying vector bundle
$E\rightarrow M$, with curvature its $\Omega$ and $\Omega_{k}$
the curvature of the $k$-th associated bundle, then
\[
s\left(\xi\right)=\sum_{j=1}^{m}\alpha_{j}\left(\prod_{t=1}^{m_{j}}c_{i_{j,t}}\left(\Omega_{k_{j,t}}\right)\right)\in H^{*}\left(M\right).
\]

At this point it is worth mentioning that $H^{*}\left(B_{\mathrm{com}}G\right)$
is in general not a polynomial algebra. For example when $G=U\left(n\right)$,
the TC Chern classes are not algebraically independent. However, the
relationships governing them are rather complicated. As such, their
values in a given TC structure can vary significantly. We see an example
of this in the next chapter. 
\begin{rem}
The concepts developed in this section can also be applied to vector
bundles on the quaternions. In this case, the structural group is
the simplectic group $\mathrm{Sp}\left(n\right)$. The main ideas
we needed to developed TC characteristic classes also hold for this
group. As we saw in the previous section we also have power maps on
cohomology, and $H^{*}\left(B_{\mathrm{com}}\mathrm{Sp}\left(n\right),\mathbb{R}\right)$
is also generated by as an algebra by $\left\{ \Phi^{k}\circ\iota\left(c_{i}\right)\mid1\leq i\leq n,k\in\mathbb{Z}\setminus\left\{ 0\right\} \right\} $.
Where again $\iota:H^{*}\left(B\mathrm{Sp}\left(n\right),\mathbb{R}\right)\rightarrow H^{*}\left(B_{\mathrm{com}}\mathrm{Sp}\left(n\right),\mathbb{R}\right)$
is induced by the natural inclusion $B_{\mathrm{com}}\mathrm{Sp}\left(n\right)\rightarrow B\mathrm{Sp}\left(n\right)$.
Also, since $\mathrm{Sp}\left(n\right)$ is a compact group, the Chern-Weil
homomorphism is in fact an isomorphism. Thus, most of the ideas we
used through out this section apply to this case as well. 
\end{rem}

\medskip{}

\section{Example }

In this final section we exhibit explicit calculations of examples
using Chern-Weil theory to compute TC characteristic classes. In particular
we show there is a TC structure $\xi$ over a 4 sphere such that $c_{i}\left(\xi\right)=0$
for every $i\in\mathbb{N}$ while $c_{2}^{-1}\left(\xi\right)\neq0$.
This shows that a TC Chern class $c_{i}^{n}$ does not necessarily
determines another TC Chern class $c_{i}^{m}$, if $m\neq n$. This
confirms that the underlying vector bundle of a TC structure does
not determine completely the TC structure.

For this we develop the computations that allow us to obtain Chern
classes in terms of clutching functions, with $SU\left(2\right)$
as the structural group. This treatment is based on the concepts presented
in \cite{key-13}, Chapter 5. 

\subsection{Connection for a vector bundle with a two sets cover with trivializations: }

Let $\pi:E\rightarrow M$ denote a smooth vector bundle over $\mathbb{C}$
of dimension $n$, with $M$ a closed manifold. Assume we can find
an open cover $\left\{ U_{1},U_{2}\right\} $ of $M$ together with
trivializations 
\begin{align*}
\varphi_{i}:\pi^{-1}\left(U_{i}\right) & \rightarrow U_{i}\times\mathbb{C}^{n}\\
e & \mapsto\left(\pi\left(e\right),h_{i}\left(e\right)\right).
\end{align*}
Suppose these trivializations have structure group a Lie group of
matrices $G$. This is, we have a function $\rho:U_{1}\cap U_{2}\rightarrow G\subseteq GL_{n}\left(\mathbb{C}\right)$
characterized by
\begin{align*}
\varphi_{2}\circ\varphi_{1}^{-1}:U_{1}\cap U_{2}\times\mathbb{C}^{n} & \rightarrow U_{1}\cap U_{2}\times\mathbb{C}^{n}\\
\left(x,v\right) & \mapsto\left(x,\rho\left(x\right)v\right).
\end{align*}
 These trivializations induce smooth sections
\begin{align*}
s_{ij}:U_{i} & \rightarrow\pi^{-1}\left(U_{i}\right)\\
x & \mapsto\varphi_{j}^{-1}\left(x,e_{j}\right),
\end{align*}
where $e_{j}$ is the $j$-th vector of the standard basis of $\mathbb{C}^{n}$,
and $i=1,2$. This setting implies that for $x\in U_{i}$the set $\left\{ s_{i1}\left(x\right),s_{i2}\left(x\right),\ldots,s_{in}\left(x\right)\right\} \subset\pi^{-1}\left(x\right)$
is a basis. Under these conditions, for a point $x\in U_{1}\cap U_{2}$
it follows that 
\begin{equation}
s_{1k}\left(x\right)=\sum_{l=1}^{n}\rho_{lk}\left(x\right)s_{2l}\left(x\right)\label{eq:sections and cocycles}
\end{equation}
where we take $\rho=\left[\rho_{lk}\right]_{k,l=1}^{n}$. 

Now let $\left\{ f_{1},f_{2}\right\} $ be a partition of the unity
subordinated to $\left\{ U_{1},U_{2}\right\} $, as well as the trivial
connections over each $U_{i}$, $\nabla^{i}$. We can now define the
connection
\[
\nabla_{X}s:=f_{1}\nabla_{X}^{1}s+f_{2}\nabla_{X}^{2}s.
\]
This means that for a vector field $X$ and a section $s$, we consider
their restriction to $U_{i}$ in order to evaluate $\nabla_{X}^{i}$.
That is, we need first to consider the decomposition of $s$ in terms
of the basis $\left\{ s_{i1},\ldots,s_{in}\right\} $, which means
that there are smooth functions $\alpha_{i}^{j}:U_{i}\rightarrow\mathbb{C}$
such that for $x\in U_{i}$
\[
s\mid_{U_{i}}\left(x\right)=\sum_{j=1}^{n}\alpha_{j}^{i}\left(x\right)s_{ij}\left(x\right).
\]
Then, applying the product rule and the definition, we have 
\[
\nabla_{X}s:=\sum_{i,j}f_{i}X\left(\alpha_{j}^{i}\right)s_{ij}.
\]

Recall that with $n$-linearly independent sections $\left\{ s_{1},\ldots,s_{n}\right\} $
we have the local expressions for both the connection and the curvature,
$R:\mathfrak{X}\left(M\right)\times\mathfrak{X}\left(M\right)\rightarrow\Gamma\left(E\right)$.
There exists 1-forms $\omega_{ij}$ and two forms $\Omega_{ij}$ such
that we can write
\[
\nabla_{X}s_{i}=\sum_{i}\omega_{ij}\left(X\right)s_{j}
\]
and 
\[
R\left(X,Y\right)\left(s_{i}\right)=\sum_{j}\Omega_{ij}\left(X,Y\right)s_{j}
\]
which gives rise to the local connection and curvature matrices
\[
\omega:=\left[\omega_{ij}\right]\:\mathrm{and}\:\Omega:=\left[\Omega_{ij}\right].
\]
 These local forms are related to the transition function in the following
way. From Equality \ref{eq:sections and cocycles} we get that 
\[
\nabla_{X}s_{1k}=\sum_{l=1}^{n}f_{2}X\left(\rho_{lk}\left(x\right)\right)s_{2l}\left(x\right).
\]
From differential geometry we know that for a function $f:M\rightarrow\mathbb{R}$,
$X\left(f\right)=df\left(X\right)$ holds, where $d$ is the external
derivation. Thus we get the expression 
\[
\nabla_{X}s_{1k}=\sum_{l=1}^{n}f_{2}d\left(\rho_{lk}\left(x\right)\right)\left(X\right)s_{2l}\left(x\right),
\]
which allow us to write 
\[
\nabla s_{1k}=\sum_{l=1}^{n}f_{2}d\left(\rho_{lk}\left(x\right)\right)s_{2l}\left(x\right).
\]
 By the properties of cocycles we also know that 
\[
s_{2l}\left(x\right)=\sum_{t=1}^{n}\rho_{tl}^{-1}\left(x\right)s_{1t}\left(x\right),
\]
where by $\rho_{tl}^{-1}$we mean the components of the matrix $\rho^{-1}$.
Thus, we can write 
\[
\nabla s_{1k}=\sum_{t=1}^{n}\left(f_{2}\left(\sum_{l=1}^{n}\rho_{tl}^{-1}\left(x\right)d\left(\rho_{lk}\left(x\right)\right)\right)\right)s_{1t}.
\]
By comparing this expression with the local form, we conclude that
\begin{equation}
\omega^{1}=f_{2}\rho^{-1}d\rho.\label{eq:Conn. in terms of cocycles}
\end{equation}

Our next step is to obtain the local form of the curvature. For this
we use the structural equation (see \cite{key-13} Theorem 5.21.)
\[
\Omega^{i}=d\omega^{i}+\omega^{i}\wedge\omega^{i}.
\]
Consider the equality $\rho^{-1}\rho=I$. An application of the product
rule allow us to write:
\[
0=dI=d\left(\rho^{-1}\right)\rho+\rho^{-1}d\rho.
\]
This in turn implies that 
\[
d\left(\rho^{-1}\right)\rho=-\rho^{-1}d\rho\Rightarrow d\left(\rho^{-1}\right)=-\rho^{-1}\left(d\rho\right)\rho^{-1}.
\]
Since $dd=0$, we obtain $d\left(\rho^{-1}d\rho\right)=d\left(\rho^{-1}\right)\wedge d\left(\rho\right)$,
which allow us to conclude that 
\[
d\omega^{1}=\left(\left(df_{2}\right)\rho^{-1}d\rho-f_{2}\rho^{-1}d\rho\wedge\rho^{-1}d\rho\right).
\]

On the other hand 
\[
\omega^{1}\wedge\omega^{1}=\left(f_{2}\rho^{-1}d\rho\right)\wedge\left(f_{2}\rho^{-1}d\rho\right)=f_{2}^{2}\rho^{-1}d\rho\wedge\rho^{-1}d\rho,
\]
which finally gives us 
\begin{equation}
\Omega^{1}=\left(\left(df_{2}\right)\rho^{-1}d\rho-f_{2}\rho^{-1}d\rho\wedge\rho^{-1}d\rho\right)+f_{2}^{2}\rho^{-1}d\rho\wedge\rho^{-1}d\rho.\label{eq:Curvature local form}
\end{equation}
Observe that in a point $x\notin U_{1}\cap U_{2}$, $\Omega^{1}$
is zero since the closure of the support of $f_{2}$ is contained
in $U_{2}$. Similarly, an analogue formula can be deduce for the
local form of the curvature in $U_{2}$, and deduced that it is also
zero outside $U_{1}\cap U_{2}$. Thus, we can conclude that

\medskip{}

\begin{prop}
Let $\pi:E\rightarrow M$ be a smooth vector bundle with $\left\{ U_{1},U_{2}\right\} $
an open cover of $M$, both having trivializations of $E$, $\varphi_{1}$
and $\varphi_{2}$, respectively. Let $\left\{ f_{1},f_{2}\right\} $
be a partition of unity associated to $\left\{ U_{1},U_{2}\right\} $,
respectively. If $\rho$ is the transition function associated to
$\varphi_{2}\circ\varphi_{1}^{-1}$, then the curvature $\Omega_{k}$
of the $k$-th associated bundle is given by
\[
\left(\Omega_{k}\right)_{x}=\begin{cases}
\left(\Omega_{k}^{1}\right)_{x} & x\in U_{1}\cap U_{2}.\\
0 & x\notin U_{1}\cap U_{2}.
\end{cases}
\]
Where 
\begin{equation}
\Omega_{k}^{1}=\left(df_{2}\right)\rho^{-k}d\left(\rho^{k}\right)+\left(f_{2}^{2}-f_{2}\right)\rho^{-k}d\left(\rho^{k}\right)\wedge\rho^{-k}d\left(\rho^{k}\right)\label{eq:Curvature for 2 sets covers}
\end{equation}
is the local expression on $U_{1}$.
\end{prop}

\begin{proof}
Since the $k$-th associated vector bundle has the same cover associated
to its TC structure, with transition functions equal to $\rho^{k}$,
the previous discussion provides a proof of the theorem. 
\end{proof}
It is worth mentioning that it is possible to deduce a similar formula
to (\ref{eq:Curvature local form}) for a arbitrary number of sets
in an open cover, but we will not need this. 

\subsection{Second Chern class for clutching functions with values on $SU\left(2\right)$.}

Suppose that we have a vector bundle $p:E\rightarrow M$ in such a
way that we can find an open cover $\left\{ U_{1},U_{2}\right\} $
of $M$ together with a transition function $\rho:U_{1}\cap U_{2}\rightarrow SU\left(2\right)$.
First, we are going to compute the determinant of the curvature form
in terms of the components of the matrices in $SU\left(2\right)$,
\[
SU\left(2\right):=\left\{ \left[\begin{array}{cc}
z & -\bar{w}\\
w & \bar{z}
\end{array}\right]\mid\left|z\right|^{2}+\left|w\right|^{2}=1\right\} .
\]
So let us take 
\[
\rho=\left[\begin{array}{cc}
z & -\bar{w}\\
w & \bar{z}
\end{array}\right],
\]
for which we want to compute the curvature 
\[
\Omega=\left(df_{2}\right)\rho^{-1}d\left(\rho\right)+\left(f_{2}^{2}-f_{2}\right)\rho^{-1}d\left(\rho\right)\wedge\rho^{-1}d\left(\rho\right).
\]
 Since $z\bar{z}+w\bar{w}=1$, we get by differentiating that 
\[
0=\left(\bar{z}dz+\bar{w}dw\right)+\left(zd\bar{z}+wd\bar{w}\right)\Rightarrow zd\bar{z}+wd\bar{w}=-\left(\bar{z}dz+\bar{w}dw\right)
\]
and so we have 
\[
\tau:=\rho^{-1}d\rho=\left[\begin{array}{cc}
\bar{z}dz+\bar{w}dw & \bar{w}d\bar{z}-\bar{z}d\bar{w}\\
-wdz+zdw & -\left(\bar{z}dz+\bar{w}dw\right)
\end{array}\right].
\]
Now take $\theta:=\rho^{-1}d\rho\wedge\rho^{-1}d\rho$. Using that
$\tau_{22}=-\tau_{11}$ and $\left|z\right|^{2}+\left|w\right|^{2}=1$
we get that 
\[
\theta=\left[\begin{array}{cc}
\left(\bar{w}d\bar{z}-\bar{z}d\bar{w}\right)\wedge\left(-wdz+zdw\right) & 2d\bar{z}\wedge d\bar{w}\\
-2dz\wedge dw & -\left(\bar{w}d\bar{z}-\bar{z}d\bar{w}\right)\wedge\left(-wdz+zdw\right)
\end{array}\right].
\]

which is the same as expressing it as
\[
\theta=\left[\begin{array}{cc}
\tau_{12}\wedge\tau_{21} & \theta_{12}\\
\theta_{21} & -\tau_{12}\wedge\tau_{21}
\end{array}\right].
\]
Now, by making $f:=f_{2}$ and $g:=\left(f-1\right)f$ we may express
the curvature as
\[
\Omega=df\tau+g\theta=\left[\begin{array}{cc}
df\tau_{11}+g\tau_{12}\wedge\tau_{21} & df\tau_{12}+g\theta_{12}\\
df\tau_{21}+g\theta_{21} & -\left(df\tau_{11}+g\tau_{12}\wedge\tau_{21}\right)
\end{array}\right],
\]
and its determinant is then given by 
\begin{align*}
\mathrm{det}\left(\Omega\right)= & -\left(df\tau_{11}+g\tau_{12}\wedge\tau_{21}\right)\wedge\left(df\tau_{11}+g\tau_{12}\wedge\tau_{21}\right)\\
 & -\left(df\tau_{21}+g\theta_{21}\right)\wedge\left(df\tau_{12}+g\theta_{12}\right).
\end{align*}

In order to reduce this expression, we recall that the wedge product
of a one form with itself is zero. Also, one forms commute with two
forms, so we get: 
\[
\mathrm{det}\left(\Omega\right)=-g^{2}\theta_{12}\wedge\theta_{21}-gdf\wedge\left(\tau_{12}\wedge\theta_{21}+\tau_{21}\wedge\theta_{12}+2\tau_{11}\wedge\tau_{12}\wedge\tau_{21}\right).
\]

By recalling that $\tau_{11}\wedge\tau_{12}=d\bar{z}\wedge d\bar{w}$
we get: 
\[
\tau_{11}\wedge\tau_{12}\wedge\tau_{21}=-\left(wdzd\bar{z}d\bar{w}+zd\bar{z}dwd\bar{w}\right),
\]
\[
\tau_{12}\wedge\theta_{21}=2\left(\bar{z}dzdwd\bar{w}+\bar{w}dzd\bar{z}dw\right)
\]
and
\[
\tau_{21}\wedge\theta_{12}=-2\left(zd\bar{z}dwd\bar{w}+wdzd\bar{z}d\bar{w}\right).
\]
Now take
\[
A:=\tau_{12}\wedge\theta_{21}+\tau_{21}\wedge\theta_{12}+2\tau_{11}\wedge\tau_{12}\wedge\tau_{21}
\]
then 
\begin{align*}
A= & 2\left[\left(\bar{z}dzdwd\bar{w}+\bar{w}dzd\bar{z}dw\right)-\left(zd\bar{z}dwd\bar{w}+wdzd\bar{z}d\bar{w}\right)\right.\\
 & \left.-\left(wdzd\bar{z}d\bar{w}+zd\bar{z}dwd\bar{w}\right)\right]
\end{align*}

which gives us 
\[
A=2\left(\bar{z}dzdwd\bar{w}+\bar{w}dzd\bar{z}dw-2\left(zd\bar{z}dwd\bar{w}+wdzd\bar{z}d\bar{w}\right)\right)
\]
which we can now replace to have 
\begin{equation}
\det\left(\Omega\right)=4\left(f_{2}-1\right)^{2}f_{2}^{2}dzd\bar{z}dwd\bar{w}-\left(f_{2}-1\right)f_{2}df_{2}\wedge A.\label{eq:Det of the curvature for SU(2)}
\end{equation}

Now we are going to use this formula to find the second Chern class
in terms of a smooth Clutching function $\varphi:S^{3}\rightarrow SU\left(2\right)$
(see \cite{key-2}, Chapter 1). Consider the sets
\[
S^{4}=\left\{ \mathbf{x}=\left(x_{1},\ldots,x_{5}\right)\in\mathbb{R}^{5}\mid\left\Vert \mathbf{x}\right\Vert =1\right\} ,
\]
\[
D_{+}=\left\{ \left(x_{1},\ldots,x_{5}\right)\in S^{4}\mid x_{5}\geq0\right\} ,
\]
\[
D_{-}=\left\{ \left(x_{1},\ldots,x_{5}\right)\in S^{4}\mid x_{5}\leq0\right\} 
\]
and the open set 
\begin{align*}
V= & \left\{ \left(x_{1},\ldots,x_{5}\right)\in S^{4}\mid\right.\\
 & \left.-1/3<x_{5}<1/3\right\} .
\end{align*}
Also let 
\[
U_{1}:=D_{+}\cup V\:\mathrm{and}\:U_{2}:=D_{-}\cup V
\]
 and identify $S^{3}$ with the equator $\left\{ \left(x_{1},\ldots,x_{5}\right)\in S^{4}\mid x_{5}=0\right\} $. 

Using \textquotedbl bump\textquotedbl{} functions we can obtain a
partition of the unity $f_{1},f_{2}:S^{4}\rightarrow\left[0,1\right]$
such that they depend only on the \textquotedbl height\textquotedbl{}
$x_{5}$ and $f_{i}\mid_{U_{i}\setminus V}\equiv1$. Also the clutching
function $\varphi:S^{3}\rightarrow SU\left(2\right)$ can be composed
with a smooth \textquotedbl perpendicular\textquotedbl{} retraction
of $V$ to $S^{3}$, to obtain a transition function $\rho:V\rightarrow SU\left(2\right)$
independent of $x_{5}$.

Under this conditions is clear that
\begin{itemize}
\item $df_{2}=\frac{\partial f_{2}}{\partial r}dr$, and
\item If 
\[
\rho=\left[\begin{array}{cc}
z & -\bar{w}\\
w & \bar{z}
\end{array}\right]
\]
any four form depending on $z$,$\bar{z}$, $w$ and $\bar{w}$ is
zero, since these functions depend only on three variables. 
\end{itemize}
We are in position to apply the previous results to obtain that 
\[
\det\left(\Omega\right)=4\left(f_{2}-1\right)^{2}f_{2}^{2}dzd\bar{z}dwd\bar{w}-\left(f_{2}-1\right)f_{2}df_{2}\wedge A.
\]
Where $A=2\left(\bar{z}dzdwd\bar{w}+\bar{w}dzd\bar{z}dw-2\left(zd\bar{z}dwd\bar{w}+wdzd\bar{z}d\bar{w}\right)\right)$.
However, by construction we have that $dzd\bar{z}dwd\bar{w}=0$ and
so 
\[
\int_{S^{4}}\det\left(\Omega\right)=\left(\int_{-1}^{1}\left(\left(1-f_{2}\right)f_{2}\frac{\partial f_{2}}{\partial r}\right)dr\right)\int_{S^{3}}A.
\]
First, notice that by construction it follows that 
\[
\int_{-1}^{1}\left(\left(1-f_{2}\right)f_{2}\frac{\partial f_{2}}{\partial r}\right)dr=-\frac{1}{6}.
\]
Finally since the second Chern class in this case is the determinant
of the curvature times $\left(\frac{i}{2\pi}\right)^{2}$, we get

\medskip{}

\begin{prop}
The second Chern class associated to a clutching function $\varphi:S^{3}\rightarrow SU\left(2\right)$
is given by 
\[
c_{2}=\frac{1}{24\pi^{2}}\int_{S^{3}}A.
\]
Here $A$ is a 3-form given by 
\[
2\left(\bar{z}dzdwd\bar{w}+\bar{w}dzd\bar{z}dw-2\left(zd\bar{z}dwd\bar{w}+wdzd\bar{z}d\bar{w}\right)\right)
\]
and the functions $z,w:S^{3}\rightarrow SU\left(2\right)$ are determined
by the clutching function, $\varphi=\left[\begin{array}{cc}
z & -\bar{w}\\
w & \bar{z}
\end{array}\right].$
\end{prop}

\subsection{A non trivial TC structure over a trivial vector bundle }

It is already known that there are trivial vector bundles with non
trivial TC structures over them. In this section we are going to use
such a structure to show that: 
\begin{thm}
There exists a TC structure 
\[
\xi=\left\{ E\rightarrow S^{4},\left\{ U_{1},U_{2},U_{3}\right\} ,\rho_{ij}:U_{i}\cap U_{j}\rightarrow SU\left(2\right)\right\} 
\]
 such that $E\rightarrow S^{4}$ is a trivial bundle, and such that
$c_{2}^{-1}\left(\xi\right)=-1$, implying that the TC structure is
non trivial.
\end{thm}

This in particular highlights how the TC characteristic classes depends
on the TC structure and not on the equivalence class of their underlying
bundle. 

Now, to prove this theorem we are based on the construction made by
D. Ramras and B. Villareal (\cite{key-15}, Chapter 3). In what follows,
we first define the vector bundle by defining an open cover on $S^{4}$
and transition functions on them. This defines a TC structure 
\[
\xi=\left\{ E\rightarrow S^{4},\left\{ U_{1},U_{2},U_{3}\right\} ,\rho_{ij}:U_{i}\cap U_{j}\rightarrow SU\left(2\right)\right\} .
\]
Then by considering the $\left(-1\right)$-powers of these transition
functions we also obtain the $\left(-1\right)$-th associated bundle,
$E^{-1}$.

Next we are going to use Lemma 3.1 of \cite{key-15} to show that
both vector bundles obtained can be described, up to isomorphism,
by a given clutching functions. Then on one hand by showing that the
clutching function associated to $E$ is trivial, we conclude that
$E$ is trivial. On the other hand, we use the clutching function
associated to $E^{-1}$ together with the formulas of the previous
sections, to conclude that $c_{2}^{-1}\left(\xi\right)=-1$.

We outline how their initial construction can be made in the smooth
category, which allows us to reduce the problem of computing the Chern
class by using Clutching functions. 

We are constructing a TC structure on a vector bundle defined over
$S^{4}$ in terms of a triple open cover $\left\{ U_{1},U_{2},U_{3}\right\} $
and transition functions between them. These transition functions
themselves will be described in terms of two functions 
\[
\rho_{1},\rho_{2}:D_{3}\rightarrow SU\left(2\right),
\]
where $D_{3}$ is the 3-dimensional closed disk of radius 1. 

For this, take 
\[
S^{4}=\left\{ \mathbf{x}=\left(x_{1},\ldots,x_{5}\right)\in\mathbb{R}^{5}\mid\left\Vert \mathbf{x}\right\Vert =1\right\} 
\]
and for $1/5>\epsilon>0$ consider the triple open cover 
\[
U_{1}:=\left\{ \left(x_{1},\ldots,x_{5}\right)\in S^{4}\mid x_{5}>-\epsilon\right\} ,
\]
\[
U_{2}:=\left\{ \left(x_{1},\ldots,x_{5}\right)\in S^{4}\mid x_{5}<0,x_{4}>-\epsilon\right\} 
\]
and
\[
U_{3}:=\left\{ \left(x_{1},\ldots,x_{5}\right)\in S^{4}\mid x_{5}<0,x_{4}<\epsilon\right\} .
\]
Also call $D_{-}=\left\{ \left(x_{1},\ldots,x_{5}\right)\in S^{4}\mid x_{5}\leq0\right\} $
and identify the closed 3-dimensional disk with
\[
D_{3}=\left\{ \left(x_{1},\ldots,x_{5}\right)\in S^{4}\mid x_{5}\leq0,x_{4}=0\right\} .
\]
There is a natural retraction $r:D_{-}\rightarrow D_{3}$ leaving
$D_{3}$ fixed (See Figure \ref{fig:Retraction}). This is a smooth
function almost everywhere.

\begin{figure}
\includegraphics[scale=0.4]{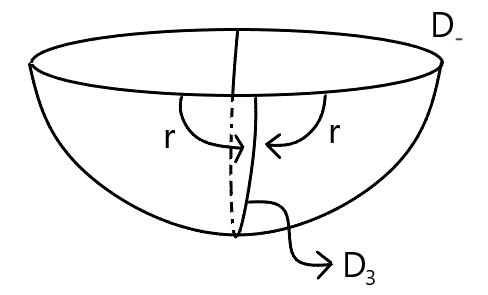}

\caption{\label{fig:Retraction}Retraction $r:D_{-}\rightarrow D_{3}$}
\end{figure}

Take $V=D_{3}\cap U_{1}$. Then we get that 
\[
V=\left\{ \left(x_{1},\ldots,x_{5}\right)\in D_{3}\mid x_{5}>-1/3\right\} .
\]
 Now suppose that the functions $\rho_{1},\rho_{2}:D_{3}\rightarrow SU\left(2\right)$
are smooth functions such that:
\begin{itemize}
\item They are independent of the radius in $D_{3}$ in $V$.
\item They are commutative in the closure of $V$. 
\end{itemize}
We define the transition functions $\rho_{ij}:U_{i}\cap U_{j}\rightarrow SU\left(2\right)$
by
\begin{itemize}
\item $\rho_{12}:=\rho_{1}\circ r.$
\item $\rho_{23}:=\rho_{2}\circ r.$
\item $\rho_{13}:=\left(\rho_{1}\circ r\right)\left(\rho_{2}\circ r\right).$
\end{itemize}
Since $r\left(U_{1}\cap U_{2}\cap U_{3}\right)\subseteq V$ by construction,
the previous cocycles commute with each other in their common domain
$U_{1}\cap U_{2}\cap U_{3}$. This transition functions allow us to
construct a smooth vector bundle $E\rightarrow S^{4}$, and so we
have constructed a TC structure 
\[
\xi=\left\{ E\rightarrow S^{4},\left\{ U_{1},U_{2},U_{3}\right\} ,\rho_{ij}:U_{i}\cap U_{j}\rightarrow SU\left(2\right)\right\} .
\]

\subsection{Associated clutching functions:}

Before dealing with the result we need, it is important to highlight
the following. Suppose $E_{1}\rightarrow M$ and $E_{2}\rightarrow M$
are smooth vector bundles with classifying functions $f_{i}:M\rightarrow BSU\left(n\right)$,
$i=1,2$. If there is a (non necessarily continous) homotopy between
$f_{1}$ and $f_{2}$, and there is class $c\in H^{*}\left(BSU\left(n\right)\right)$,
it follows that $f_{1}^{*}\left(c\right)=f_{2}^{*}\left(c\right)\in H^{*}\left(M\right)$.
Now consider the curvatures $\Omega_{1}$ and $\Omega_{2}$ for $E_{1}$
and $E_{2}$, respectively. By the Chern-Weil isomorphism, we get
that $c\left(\Omega_{1}\right)=f_{1}^{*}\left(c\right)$ and $c\left(\Omega_{2}\right)=f_{2}^{*}\left(c\right)$,
and thus $c\left(\Omega_{1}\right)=c\left(\Omega_{2}\right)$. In
particular if there is a continuous (but not smooth) isomorphism of
vector bundles between $E_{1}$ and $E_{2}$, their classifying functions
will be homotopic and their characteristic classes will coincide. 

Now consider the closed sets
\[
C_{1}:=\left\{ \left(x_{1},\ldots,x_{5}\right)\in S^{4}\mid x_{5}\geq0\right\} ,
\]
\[
C_{2}:=\left\{ \left(x_{1},\ldots,x_{5}\right)\in S^{4}\mid x_{5}\leq0,x_{4}\geq0\right\} 
\]
and
\[
C_{2}:=\left\{ \left(x_{1},\ldots,x_{5}\right)\in S^{4}\mid x_{5}\leq0,x_{4}\leq0\right\} .
\]

It is clear that there is a retraction $r_{i}:U_{i}\rightarrow C_{i}$
leaving $C_{i}$fixed, for $i=1,2,3$. Notice that by applying on
$U_{2}\cap U_{3}$ $r_{2}$ first and then $r_{3}$, we obtain a retraction
$r_{23}:U_{2}\cap U_{3}\rightarrow C_{2}\cap C_{3}$ leaving $C_{2}\cap C_{3}$
fixed. For $U_{1}\cap U_{2}$ we apply first $r_{2}$ and then $r_{3}$,
we obtain a retraction $r_{12}:U_{1}\cap U_{2}\rightarrow C_{1}\cap C_{2}$
leaving $C_{1}\cap C_{2}$ fixed, and similarly we obtain $r_{13}:U_{1}\cap U_{3}\rightarrow C_{1}\cap C_{3}$
leaving $C_{1}\cap C_{3}$ fixed. Via this restrictions of $\rho_{ij}$
we obtain transition functions for the closed cover $\left\{ C_{1},C_{2},C_{3}\right\} $:
\[
\tilde{\rho}_{ij}:C_{i}\cap C_{j}\rightarrow SU\left(2\right).
\]
This new transition functions are clearly homotopic to $\rho_{ij}$
via the retractions $r_{ij}$. Thus, they characterized vector bundles
over $S^{4}$ whose classifying functions are homotopic. 

Consider the identification $S^{3}\cong\left\{ \left(x_{1},\ldots,x_{5}\right)\in S^{4}\mid x_{5}=0\right\} $.
This setting allow us to apply Lemma 3.1 of \cite{key-15}. There
they show that the bundle induced by these three cocycles is isomorphic
to the vector bundle with clutching function $\varphi:S^{3}\rightarrow SU\left(2\right)$
defined for $\mathbf{x}=\left(x_{1},\ldots,x_{5}\right)$ by 
\[
\varphi\left(\mathbf{x}\right):=\begin{cases}
\rho_{1}\left(r\left(\mathbf{x}\right)\right)\rho_{2}\left(r\left(\mathbf{x}\right)\right) & x_{4}\geq0.\\
\rho_{1}\left(r\left(\mathbf{x}\right)\right)\rho_{2}\left(r\left(\mathbf{x}\right)\right) & x_{4}\leq0.
\end{cases}
\]
The function $\varphi$ can clearly be extended continuously to the
whole disk $D_{-}$, since we defined $r$ on $D_{-}$. This implies
that $\varphi$ is null homotopic, and thus, the vector bundle given
by these cocycles is trivial. 

Now lets consider the same construction but using the cocycles given
by $\sigma_{ij}=\rho_{ij}^{-1}$. They give rise to the $\left(-1\right)$-th
associated bundle by definition. Once again allow us to use Lemma
3.1 of \cite{key-15}. We conclude that this bundle can be obtain,
up to isomorphims, by the clutching function given by
\[
\phi\left(y\right):=\begin{cases}
\rho_{1}^{-1}\left(r\left(\mathbf{x}\right)\right)\rho_{2}^{-1}\left(r\left(\mathbf{x}\right)\right) & x_{4}\geq0.\\
\rho_{2}^{-1}\left(r\left(\mathbf{x}\right)\right)\rho_{1}^{-1}\left(r\left(\mathbf{x}\right)\right) & x_{4}\leq0.
\end{cases}
\]
In this case this function cannot be extended continuously to $D_{-}$
if $\rho_{1}$ and $\rho_{2}$ do not commute everywhere in $D_{3}$.
So $\phi$ is not necessarily null homotopic. 

\subsection{Existence of a non trivial TC structure:}

From the previous part, we need to show that it is possible to obtain
a non null homotopic clutching function $\phi$. For this it is enought
to display two functions $\rho_{1},\rho_{2}:D_{3}\rightarrow SU\left(2\right)$
such that they commute in $\partial D_{3}\cong S^{3}$, giving us
a non zero Chern class for the bundle with clutching function $\phi:S^{3}\rightarrow SU\left(2\right)$. 

We can describe $\phi$ in terms of the northern and southern hemispheres
of $S^{3}$, $D_{+}$ and $D_{-}$, respectively. Each of them can
be identify with the 3-dimensional disc $D_{3}$. Then we get that
\[
\phi\left(y\right):=\begin{cases}
\rho_{1}^{-1}\rho_{2}^{-1} & \mathrm{in}\,D_{+},\\
\rho_{2}^{-1}\rho_{1}^{-1} & \mathrm{in}\,D_{-}.
\end{cases}
\]
 For brevity allow us to write the matrices of $SU\left(2\right)$
as 
\[
\left(a,b\right):=\left[\begin{array}{cc}
a & -\bar{b}\\
b & \bar{a}
\end{array}\right].
\]

\medskip{}

\begin{prop}
Consider $D_{3}$ under spherical coordinates and take 
\[
\rho_{1}\left(\alpha,\beta,r\right):=\begin{cases}
\left(\sin\left(\frac{\pi}{2}r\right)e^{i\alpha},\cos\left(\frac{\pi}{2}r\right)\right), & 0\leq\beta\leq\pi/2.\\
\left(\sin\left(r\beta\right)e^{i\alpha},\cos\left(r\beta\right)\right) & \pi/2\leq\beta\leq\pi.
\end{cases}
\]
\[
\rho_{2}\left(\alpha,\beta,r\right):=\begin{cases}
\left(-\cos\left(\pi r\right)e^{2i\beta},\sin\left(\pi r\right)\right), & 0\leq\beta\leq\pi/2.\\
\left(\cos\left(\pi r\right),\sin\left(\pi r\right)\right) & \pi/2\leq\beta\leq\pi.
\end{cases}
\]
then the second Chern class of $\phi$ is $c_{2}\left(\phi\right)=-1$.
\end{prop}

\begin{proof}
Recalled from the previous section that if we make $\phi=\left(z,w\right)$,
the second Chern class of $\phi$ is then given by
\[
c_{2}=\frac{1}{24\pi^{2}}\int_{S^{3}}A.
\]
Where $A$ is a three form given by 
\[
2\left(\bar{z}dzdwd\bar{w}+\bar{w}dzd\bar{z}dw-2\left(zd\bar{z}dwd\bar{w}+wdzd\bar{z}d\bar{w}\right)\right).
\]
We can split this integral as
\[
\int_{S^{3}}A=\int_{D_{-}}A+\int_{D_{+}}A.
\]
Now, call $\rho_{1}^{-1}\rho_{2}^{-1}=\left(z_{1},w_{1}\right)$ and
$\rho_{2}^{-1}\rho_{1}^{-1}=\left(z_{2},w_{2}\right)$. Because of
orientations, we get
\[
\int_{S^{3}}A=\int_{D_{3}}A_{2}-\int_{D_{3}}A_{1}=\int_{D_{3}}\left(A_{2}-A_{1}\right)
\]
where
\[
A_{1}=2\left(\bar{z}_{1}dz_{1}dw_{1}d\bar{w}_{1}+\bar{w}_{1}dz_{1}d\bar{z}_{1}dw_{1}-2\left(z_{1}d\bar{z}_{1}dw_{1}d\bar{w}_{1}+w_{1}dz_{1}d\bar{z}_{1}d\bar{w}_{1}\right)\right)
\]
\[
A_{2}=2\left(\bar{z}_{2}dz_{2}dw_{2}d\bar{w}_{2}+\bar{w}_{2}dz_{2}d\bar{z}_{2}dw_{2}-2\left(z_{2}d\bar{z}_{2}dw_{2}d\bar{w}_{2}+w_{2}dz_{2}d\bar{z}_{2}d\bar{w}_{2}\right)\right).
\]

From this we have for $0\leq\beta\leq\pi/2$ that 
\begin{align*}
\left(z_{2},w_{2}\right)=\rho_{2}^{-1}\rho_{1}^{-1}= & \left(-\sin\left(\frac{\pi}{2}r\right)\cos\left(\pi r\right)e^{-\left(\alpha+2\beta\right)i}-\sin\left(\pi r\right)\cos\left(\frac{\pi}{2}r\right),\right.\\
 & \left.\cos\left(\pi r\right)\cos\left(\frac{\pi}{2}r\right)e^{2i\beta}-\sin\left(\pi r\right)\sin\left(\frac{\pi}{2}r\right)e^{-i\alpha}\right)
\end{align*}
\begin{align*}
\left(z_{1},w_{1}\right)=\rho_{1}^{-1}\rho_{2}^{-1}= & \left(-\sin\left(\frac{\pi}{2}r\right)\cos\left(\pi r\right)e^{-\left(\alpha+2\beta\right)i}-\sin\left(\pi r\right)\cos\left(\frac{\pi}{2}r\right),\right.\\
 & \left.\cos\left(\pi r\right)\cos\left(\frac{\pi}{2}r\right)e^{-2i\beta}-\sin\left(\pi r\right)\sin\left(\frac{\pi}{2}r\right)e^{i\alpha}\right),
\end{align*}
while for $\pi/2\leq\beta\leq\pi$ we have
\begin{align*}
\left(z_{2},w_{2}\right)=\rho_{2}^{-1}\rho_{1}^{-1}= & \left(\sin\left(r\beta\right)\cos\left(\pi r\right)e^{-\alpha i}-\sin\left(\pi r\right)\cos\left(r\beta\right),\right.\\
 & \left.-\sin\left(\pi r\right)\sin\left(r\beta\right)e^{-i\alpha}-\cos\left(\pi r\right)\cos\left(r\beta\right)\right)
\end{align*}
\begin{align*}
\left(z_{1},w_{1}\right)=\rho_{1}^{-1}\rho_{2}^{-1}= & \left(\sin\left(r\beta\right)\cos\left(\pi r\right)e^{-\alpha i}-\sin\left(\pi r\right)\cos\left(r\beta\right),\right.\\
 & \left.-\sin\left(\pi r\right)\sin\left(r\beta\right)e^{i\alpha}-\cos\left(\pi r\right)\cos\left(r\beta\right)\right).
\end{align*}
Observe that in both cases we have that $z_{1}=z_{2}$ y $w_{1}=\bar{w}_{2}$.
Then we have to integrate the form 
\[
A_{2}-A_{1}=4\underbrace{\left(2z_{1}d\bar{z}_{1}-\bar{z}_{1}dz_{1}\right)dw_{1}d\bar{w}_{1}}_{B_{1}}+6\underbrace{\left(w_{1}d\bar{w}_{1}-\bar{w}_{1}dw_{1}\right)dz_{1}d\bar{z}_{1}}_{B_{2}}.
\]
Now consider the decomposition $z_{1}=z=x+yi$ and $w_{1}=w=u+vi$,
where $x,y,u$ and $v$ are real functions. Then it follows that 
\[
\left(2z_{1}d\bar{z}_{1}-\bar{z}_{1}dz_{1}\right)=\left(xdx+ydy\right)+3\left(ydx-xdy\right)i,
\]
\[
dw_{1}d\bar{w}_{1}=-2idu\wedge dv,
\]
\[
\left(w_{1}d\bar{w}_{1}-\bar{w}_{1}dw_{1}\right)=2\left(vdu-udv\right)i
\]
and $dzd\bar{z}=-2idx\wedge dy.$ This gives us
\[
B_{1}=6\left(ydx-xdy\right)du\wedge dv-2i\left(xdx+ydy\right)du\wedge dv
\]
and 
\[
B_{2}=4\left(vdu-udv\right)dx\wedge dy.
\]

Since we only need to compute the real part of the first form, we
consider the form $6\left(ydx-xdy\right)du\wedge dv$ instead of all
of $B_{1}$. Then by definition we get 
\[
ydx-xdy=\left(y\frac{\partial x}{\partial\alpha}-x\frac{\partial y}{\partial\alpha}\right)d\alpha+\left(y\frac{\partial x}{\partial\beta}-x\frac{\partial x}{\partial\beta}\right)d\beta+\left(y\frac{\partial x}{\partial r}-x\frac{\partial y}{\partial r}\right)dr,
\]
\[
vdu-udv=\left(v\frac{\partial u}{\partial\alpha}-u\frac{\partial v}{\partial\alpha}\right)d\alpha+\left(v\frac{\partial u}{\partial\beta}-u\frac{\partial x}{\partial\beta}\right)d\beta+\left(v\frac{\partial u}{\partial r}-u\frac{\partial v}{\partial r}\right)dr,
\]
\[
du\wedge dv=\left(\frac{\partial u}{\partial\alpha}\frac{\partial v}{\partial\beta}-\frac{\partial v}{\partial\alpha}\frac{\partial u}{\partial\beta}\right)d\alpha d\beta+\left(\frac{\partial u}{\partial\alpha}\frac{\partial v}{\partial r}-\frac{\partial v}{\partial\alpha}\frac{\partial u}{\partial r}\right)d\alpha dr+\left(\frac{\partial u}{\partial\beta}\frac{\partial v}{\partial r}-\frac{\partial v}{\partial\beta}\frac{\partial u}{\partial r}\right)d\beta dr
\]
and
\[
dx\wedge dy=\left(\frac{\partial x}{\partial\alpha}\frac{\partial y}{\partial\beta}-\frac{\partial y}{\partial\alpha}\frac{\partial x}{\partial\beta}\right)d\alpha d\beta+\left(\frac{\partial x}{\partial\alpha}\frac{\partial y}{\partial r}-\frac{\partial y}{\partial\alpha}\frac{\partial x}{\partial r}\right)d\alpha dr+\left(\frac{\partial x}{\partial\beta}\frac{\partial y}{\partial r}-\frac{\partial x}{\partial\beta}\frac{\partial y}{\partial r}\right)d\beta dr.
\]
Now call $J_{1}:=\left(ydx-xdy\right)du\wedge dv$ and $J_{2}:=\left(ydx-xdy\right)du\wedge dv$.
Then we have 

\begin{align*}
J_{1}= & \left[\left(y\frac{\partial x}{\partial\alpha}-x\frac{\partial y}{\partial\alpha}\right)\left(\frac{\partial u}{\partial\beta}\frac{\partial v}{\partial r}-\frac{\partial v}{\partial\beta}\frac{\partial u}{\partial r}\right)-\left(y\frac{\partial x}{\partial\beta}-x\frac{\partial x}{\partial\beta}\right)\left(\frac{\partial u}{\partial\alpha}\frac{\partial v}{\partial r}-\frac{\partial v}{\partial\alpha}\frac{\partial u}{\partial r}\right)\right.\\
 & \left.+\left(y\frac{\partial x}{\partial r}-x\frac{\partial y}{\partial r}\right)\left(\frac{\partial u}{\partial\alpha}\frac{\partial v}{\partial\beta}-\frac{\partial v}{\partial\alpha}\frac{\partial u}{\partial\beta}\right)\right]d\alpha\wedge d\beta\wedge dr,
\end{align*}
\begin{align*}
J_{2}= & \left[\left(v\frac{\partial u}{\partial\alpha}-u\frac{\partial v}{\partial\alpha}\right)\left(\frac{\partial x}{\partial\beta}\frac{\partial y}{\partial r}-\frac{\partial x}{\partial\beta}\frac{\partial y}{\partial r}\right)-\left(v\frac{\partial u}{\partial\beta}-u\frac{\partial x}{\partial\beta}\right)\left(\frac{\partial u}{\partial\alpha}\frac{\partial v}{\partial r}-\frac{\partial v}{\partial\alpha}\frac{\partial u}{\partial r}\right)\right.\\
 & \left.+\left(v\frac{\partial u}{\partial r}-u\frac{\partial v}{\partial r}\right)\left(\frac{\partial x}{\partial\alpha}\frac{\partial y}{\partial\beta}-\frac{\partial y}{\partial\alpha}\frac{\partial x}{\partial\beta}\right)\right]d\alpha\wedge d\beta\wedge dr.
\end{align*}
Then by replacing we get $A_{2}-A_{1}=24\left(J_{1}+J_{2}\right),$
and even further 
\[
c_{2}\left(\phi\right)=\frac{1}{\pi^{2}}\int\left(J_{1}+J_{2}\right).
\]
Where by using computational software we obtain that $\int\left(J_{1}+J_{2}\right)=-\pi^{2}$,
giving us $c_{2}\left(\phi\right)=$-1.

We conclude that $c_{2}^{-1}\left(\xi\right)=-1$ for our TC structure,
implying that the TC structure is non trivial.
\end{proof}
\medskip{}


\begin{thebibliography}{Gritschacher}
\bibitem[AC]{key-16}A. Adem and F.R. Cohen. Commuting elements and
spaces of homomorphisms, 2007. Mathematische Annalen 338:587--626.

\bibitem[ACT]{key-17}A. Adem, F.R. Cohen, E Torres-Giese, Commuting
elements, simplicial spaces and filtrations of classifying spaces,
2012. Math. Proc. Cambridge Philos. Soc. 152 (2012) 91--114 MR2860418

\bibitem[AG]{key-1}A. Adem and J.M. Gómez. A classifying space for
commutativity in Lie groups, 2015. Algebraic \& Geometric Topology
15 493--535.

\bibitem[AGLT]{key-18}A. Adem, J.M. Gómez, J. Lind and U. Tillman.
Infinite loop spaces and nilpotent K--theory, 2017. Algebraic \&
Geometric Topology. 17 (2017) 869-893.

\bibitem[Baird]{key-9}T.J. Baird. Cohomology of the space of commuting
n--tuples in a compact Lie group, 2007. Algebr. Geom. Topol. 7 737--754.

\bibitem[Dupont]{key-10}J.L. Dupont. Curvature and characteristic
classes, 1978. Lecture notes in mathematics. Springer-Verlag.

\bibitem[Gritschacher]{key-14}S.P. Gritschacher. The spectrum for
commutative complex K-theory, 2018. Algebraic \& Geometric Topology,
Vol. 18, No. 2, 2018, p. 1205-1249.

\bibitem[Hatcher]{key-8}A. Hatcher. Algebraic topology, 2002. Cambridge
Univ. Press.

\bibitem[Hatcher II]{key-2}A. Hatcher. Vector bundles and K-theory,
version 2.0, 2003. Online version. 

\bibitem[Humphrey]{key-12}J.E. Humphreys. Reflection groups and Coxeter
groups, 1990. Cambridge studies in advance mathematics 29. Cambridge
University Press.

\bibitem[Husemoller]{key-3}D. Husemoller. Fiber bundles, third edition,
1994. Graduate text in mathematics 20. Springer- Verlag.

\bibitem[May]{key-6} J.P. May, Simplicial objects in algebraic topology,
1967. The university of Chicago press. 

\bibitem[Morita]{key-13}S. Morita. Geometry of differential forms,
1997. Translations of mathematical monographs. American mathematical
society.

\bibitem[RV]{key-15}D.A. Ramras and B. Villarreal. Commutative cocycles
and stable bundles over surfaces, 2019. Forum Mathematicum, 31(6),
1395--1415. https://doi.org/10.1515/forum-2018-0163.

\bibitem[Vaccarino]{key-11}F. Vaccarino. The ring of multisimmetric
functions, 2005. Annales de l'Institut Fourier, Tome 55 (2005) no.
3, pp. 717-731.
\end{thebibliography}
\end{document}